\documentclass[11pt,a4paper,reqno]{amsart} 
\usepackage{amssymb}
\usepackage{latexsym}
\usepackage{amsmath}
\usepackage{mathrsfs}
\usepackage{upgreek}
\usepackage{mathabx}
\usepackage[X2,T1]{fontenc}
\usepackage{cite}
\usepackage{color}
\usepackage{calc}                   

\newcommand{\scal}[2]{\langle #1,#2\rangle}
\newcommand{\rr}[1]{\mathbf R^{#1}}

\newcommand{\nm}[2]{\Vert #1\Vert _{#2}}
\newcommand{\NM}[2]{\left \Vert #1\right \Vert _{#2}}
\newcommand{\nmm}[1]{\Vert #1\Vert }

\newcommand{\op}{\operatorname{Op}}

\newcommand{\sets}[2]{\{ \, #1\, ;\, #2\, \} }
\newcommand{\Sets}[2]{\left \{ \, #1\, ;\, #2\, \right \} }

\newcommand{\ep}{\varepsilon}
\newcommand{\fy}{\varphi}
\newcommand{\cdo}{\, \cdot \, }

\newcommand{\vrum}{\vspace{0.1cm}}

\newcommand{\rd}{\mathbf{R} ^{d}}
\newcommand{\rdd}{\mathbf{R} ^{2d}}

\newcommand{\GL}{\mathbf{M}}

\newcommand{\maclS}{\mathcal S}

\newcommand{\mascB}{\mathscr B}

\newcommand{\mascF}{\mathscr F}
\newcommand{\mascP}{\mathscr P}
\newcommand{\mascS}{\mathscr S}

\numberwithin{equation}{section}          

\newtheorem{thm}{Theorem}
\numberwithin{thm}{section}
\newtheorem*{tom}{\rubrik}
\newcommand{\rubrik}{}
\newtheorem{prop}[thm]{Proposition}
\newtheorem{cor}[thm]{Corollary}
\newtheorem{lemma}[thm]{Lemma}

\theoremstyle{definition}

\newtheorem{defn}[thm]{Definition}
\newtheorem{example}[thm]{Example}

\theoremstyle{remark}
\newtheorem{rem}[thm]{Remark}

\newcommand\dela[1]{}

\title{Pseudo-differential calculi and entropy estimates
with Orlicz modulation spaces}

\author{Anupam Gumber}

\address{NuHAG, Faculty of Mathematics,
University of Vienna, Vienna, Austria}

\email{anupam.gumber@univie.ac.at}

\author{Nimit Rana}
\address{Department of Mathematics, Imperial College London, London, UK}
\email{n.rana22@imperial.ac.uk}

\author{Joachim Toft}

\address{Department of Mathematics,
Linn{\ae}us University, V{\"a}xj{\"o}, Sweden}

\email{joachim.toft@lnu.se}

\author{R{\"u}ya {\"U}ster}

\address{Department of Mathematics, 
Faculty of Science, {\.I}stanbul University, 
{\.I}stanbul, T{\"u}rkiye}

\email{ruya.uster@istanbul.edu.tr}

\begin{document}

\begin{abstract}
We deduce continuity properties for
pseudo-differential operators with symbols in 
Orlicz modulation spaces 
when acting on other Orlicz modulation spaces. 
In particular we extend well-known results
in the literature. For example we generalize
the classical result that
if
$$
\frac 1{p'}+\frac 1{q'}
\le
\frac 1{p_1}+\frac 1{p_2},
\quad
\frac 1{p'}+\frac 1{q'}
\le
\frac 1{q_1}+\frac 1{q_2},
\quad
p_j,q_j\le q',
\quad
q\le p
$$
and $a\in M^{p,q}$, then the
pseudo-differential operator $\op (a)$
is continuous from $M^{p_1,q_1}$
to $M^{p_2',q_2'}$.

\par

We also show that the entropy functional
$E_\phi$ possess suitable continuity properties
on a suitable Orlicz modulation space $M^\Phi$ satisfying
$M^p\subseteq M^\Phi \subseteq M^2$, though
$E_\phi$ is discontinuous on $M^2=L^2$.
\end{abstract}



\maketitle

\section{Introduction}\label{sec0}

\par

Pseudo-differential operators are important in several 
fields of science and technology. In the theory of 
partial differential equations, they are convenient 
tools for handling various kinds of problems, e.{\,}g. 
parametrix constructions, micro-local properties and 
invertibility of (hypo-)elliptic operators. In
time-frequency analysis, pseudo-differential operators 
appears when modelling non-stationary filters.

\par

A pseudo-differential operator is a rule
in which for every function or distribution
$a$ (the symbol), defined
on the phase space (or time-frequency shift
space) $\rr {2d}$ assigns a linear operator
$\op (a)$ acting on functions or distributions
defined on $\rr d$. The assumptions on
the symbols, domains and ranks for
the pseudo-differential operators, usually
resemble on structures where they are
applied. Therefore, in the theory of
partial differential operators,
one usually assumes that the symbols
are smooth and that
differentiations of the symbols
lead to more restrictive
growth/decay properties at infinity.

\par

When using pseudo-differential operators
for modelling time-dependent filters
in time-frequency analysis, any similar
assumptions on smoothness are usually
not relevant. Here it is more
relevant to assume that the involved
symbols, inputs and outputs (i.{\,}e.
the filter constants, ingoing signals 
and outgoing signals) should fulfill
conditions on translation
and modulation invariance, as
well as
certain energy estimates of the
time-frequency content. This leads
to that the involved functions
and distributions should belong to
suitable \emph{modulation spaces}
a family of functions and distribution
spaces, introduced by Feichtinger in
\cite{F1}. The theory of such spaces
was thereafter extended in several ways
(see e.{\,}g.
\cite{CorRod1,Fei5,FeiGro1,FeiGro2,GaSa,Teo1,
Toft5,Toft16,Toft28} and the references
therein).

\par

Recently, some investigations of
Orlicz modulation spaces have been performed in
\cite{TofUst,ToUsNaOz,SchF}. Such spaces are obtained by
imposing an Orlicz norm estimates on the
short time Fourier transforms of the involved
functions and distributions. By the definition it follows
that the family of Orlicz modulation spaces contain
all classical modulation spaces $M^{p,q}_{(\omega )}(\rr d)$,
introduced by Feichtinger in \cite{F1}, which essentially
follows from the fact that the family of Orlicz spaces contains
all Lebesgue spaces. (See
\cite{Hor1} and Section \ref{sec1} for notations.)
On the other hand, the Orlicz modulation spaces becomes
a subfamily of broader classes of modulation spaces, given in
e.{\,}g. \cite{Fei5,Rau1,Rau2}.

\par

A question which might appear is whether there are relevant
situations where it is fruitful to search among
Orlicz modulation spaces to deduce sharper estimates compared
to classical modulation spaces.
For example, consider the entropy functional
on short-time Fourier transforms
\begin{equation}\label{Eq:EntropyFunc}
E(f) = E_\phi (f)
\equiv
-\iint _{\rr {2d}}|V_\phi f(x,\xi )|^2
\log |V_\phi f(x,\xi )|^2\, dxd\xi
+
\nm {V_\phi f}{L^2}^2\log \nm {V_\phi f}{L^2}^2.
\end{equation}
Here
$\phi \in \mascS (\rr d)\setminus 0$ is fixed,
and as usual we set
$$
0\log 0 \equiv  \lim _{t\to 0+}t\log t =0.
$$
We recall the entropy condition
\begin{equation}\label{Eq:EntropyEst}
E_\phi (f)
\ge
d\left (
1+\log ({\textstyle{\frac \pi 2}})
\right ),
\quad \text{when}\quad
\nm f{L^2}\nm \phi {L^2}=1,
\end{equation}
which is essential in certain types of estimates of
the kinetic energy in quantum systems
(see e.{\,}g. \cite{MajLab1,MajLab2,Lie1,LieSol}
and the references therein).

\par

For the Orlicz modulation space
\begin{equation}\label{Eq:SpecOrlModSpIntro}
M^\Phi (\rr d),
\qquad \Phi (t)=-t^2\log t,\quad 0\le  t\le e^{-\frac 23}
\end{equation}
we observe that the Young function $\Phi$ resembles
with the structures of the entropy functional $E_\phi$.
A question then appear whether the space in
\eqref{Eq:SpecOrlModSpIntro} is better designed
as domain for $E_\phi$, compared to the strictly
larger space $M^2(\rr d)=L^2(\rr d)$, which is usually taken
as the domain for $E_\phi$ (cf. \cite{Lie1,LieSol}). 

\par

We also notice that $M^\Phi (\rr d)$ in \eqref{Eq:SpecOrlModSpIntro}
makes sense, while
\begin{equation}\label{Eq:NotOrliczSpaceIntro}
L^\Phi (\rr d), \qquad \Phi (t)=-t^2\log t,\qquad 0\le  t<\infty
\end{equation}
does not makes sense as an Orlicz space.
(See Theorem \ref{Thm:EntropyCont} and
Lemma \ref{Lem:OrlModCloseM2}
in Section \ref{sec3} for details.)

\medspace

In the first part of the paper we investigate
mapping properties for
pseudo-differential operators $\op (a)$
with symbols $a$ belonging to suitable
modulation spaces or Orlicz modulation
spaces, when acting on
Orlicz modulation spaces. In particular
we find suitable conditions on the
Young functions
$\Phi _j$, $\Phi$, $\Psi _j$ and
$\Psi$, $j=1,2$, in order to the
pseudo-differential operators
\begin{align}
\op (a) &: M^{\Phi _1,\Psi _1}(\rr d)
\to
M^{\Phi _2^*,\Psi _2^*}(\rr d)
\label{Eq:PsDOpOrlModContIntr1}
\intertext{and}
\op (a) &: M^{\Phi ^*,\Psi ^*}(\rr d)
\to
W^{\Phi ,\Psi}(\rr d)
\label{Eq:PsDOpOrlModContIntr2}
\end{align}
are well-defined and continuous.

\par

For example, the following two propositions are
consequences of Theorems
\ref{Thm:PseudoCont1} and \ref{Thm:Wpseudos}
in Section \ref{sec2}. Here and in what follows we
let $p'\in [1,\infty ]$ be the conjugate Lebesgue
exponent of $p\in [1,\infty ]$, i.{\,}e. $p$ and
$p'$ should satisfy $\frac 1p+\frac 1{p'}=1$,
and similarly for other Lebesgue exponents.

\par

\begin{prop}\label{Prop:PseudoContIntro1}
Let $p,q\in [1,\infty ]$ be such that $q\le p$ and $p>1$.
Also let
$\Phi _j,\Psi _j : [0,\infty ]\to [0,\infty ]$, $j=1,2$,
be such that  $t\mapsto \Phi _j(t^{\frac 1{p'}})$
and
$t\mapsto \Psi _j(t^{\frac 1{p'}})$ are
Young functions 
which fulfill the $\Delta _2$-condition,
and
\begin{alignat*}{3}
\Phi _1(t),\Phi _2(t) &\gtrsim t^{q'} & \quad
\Psi _1(t),\Psi _2(t) &\gtrsim t^{q'},
& \quad t &\ge 0,
\intertext{and}
\Phi _1^{-1}(s)\Phi _2^{-1}(s)
&\lesssim
s^{\frac 1{p'}+\frac 1{q'}}, &
\quad
\Psi _1^{-1}(s)\Psi _2^{-1}(s)
&\lesssim
s^{\frac 1{p'}+\frac 1{q'}}, &
\quad s &\ge 0.
\end{alignat*}
If $a\in M^{p,q}(\rr {2d})$,
then $\op (a)$ is continuous from
$M^{\Phi _1,\Psi _1}(\rr d)$
to
$M^{\Phi _2^*,\Psi _2^*}(\rr d)$.
\end{prop}

\par

\begin{prop}\label{Prop:WpseudosIntro}
Let $\Phi$ and $\Psi$ be Young functions
which satisfy the $\Delta _2$-condition,
and let
$a\in W^{\Psi ,\Phi}(\rr {2d})$.
Then the definition of $\op (a)$
is continuous from
$M^{\Phi ^*,\Psi ^*}(\rr d)$
to $W^{\Psi ,\Phi}(\rr
d)$.
\end{prop}

\par

More generally, we deduce weighted versions
of such continuity results, as well as
relax the assumptions on the Young
functions in such way that
they only need to fulfill a
\emph{local} $\Delta _2$-condition
near origin (see Definition \ref{Def:Delta2Cond}).
The essential ingredient for such local
condition is the fact that Orlicz
modulation spaces are completely
determined by the behaviour of the
Young functions \emph{near origin}, and the
involved weight functions. (See e.{\,}g.
\cite[Proposition 5.11]{ToUsNaOz}.)
Since Orlicz spaces contain Lebesgue spaces
as special cases, it follows that
Orlicz modulation spaces contain
the classical modulation spaces.
Hence, our results also lead
to continuity properties for
pseudo-differential operators acting
on (classical) modulation spaces.
More specific,
by choosing the Young functions in
Proposition \ref{Prop:PseudoContIntro1}
and involved weight functions in
suitable ways, our main
result Theorem \ref{Thm:PseudoCont1}
in Section \ref{sec2} include
the optimal result
\cite[Theorem 5.1]{CorNic2} by Cordero
and Nicola as special case.
In the case of unweighted spaces
\cite[Theorem 5.1]{CorNic2} attains
the following form.

\par

\begin{prop}\label{Prop:CorNicSpecCase}
Suppose that $p,p_j,q,q_j\in [1,\infty ]$, $j=1,2$,
satisfy
$$
\frac 1{p'}+\frac 1{q'} \le
\frac 1{p_1}+\frac 1{p_2},
\quad
\frac 1{p'}+\frac 1{q'} \le
\frac 1{q_1}+\frac 1{q_2},
\quad 
p_1,q_1,p_2,q_2 \le q',
\quad q\le p,
$$
and let $a\in M^{p,q}(\rr {2d})$. Then
\begin{equation}
\op (a) : M^{p_1,q_1}(\rr d)
\to
M^{p_2',q_2'}(\rr d),
\end{equation}
is continuous.
\end{prop}

\par

Proposition
\ref{Prop:CorNicSpecCase}
is a special
case of Theorem \ref{Thm:PseudoCont1}
in Section \ref{sec2}. If in addition
$p>1$, then Proposition
\ref{Prop:CorNicSpecCase}
also follows from
Proposition \ref{Prop:PseudoContIntro1}.
We also observe that for weighted (Orlicz)
modulation spaces,
Theorem \ref{Thm:PseudoCont1}
in Section \ref{sec2} permits more general
weights in the involved spaces, compared to
\cite[Theorem 5.1]{CorNic2}.

\par

There are relevant situations where
Proposition \ref{Prop:PseudoContIntro1} and
its extension Theorem \ref{Thm:PseudoCont1}
can be applied, while earlier classical results in
e.{\,}g.
\cite{CorNic,GH1,GH2,Teo1,
TeoTof,Toft5,Toft6,Toft8B,Toft28,Toft15,Toft16,Toft15B}
seem not to be applicable. For example it follows from
Proposition \ref{Prop:PseudoContIntro1} that if $p>2$,
$a\in M^{p,p'}(\rr {2d})$, then the map
\begin{equation}\label{Eq:SpecYoungFunc}
\op (a) : M^\Phi (\rr d)\to M^\Phi (\rr d),
\quad
\Phi (t) = -t^2\log t,\ t\in [0,e^{-\frac 23}],
\end{equation}
on Orlicz spaces in \eqref{Eq:SpecOrlModSpIntro},
is continuous. (See Example
\ref{Example:ContPseudo} in Section \ref{sec2}
and Remark \ref{Rem:DiscontPseudo} in Section
\ref{sec3}.) Any similar continuity property is obviously
not reachable from the investigations in
\cite{CorNic,GH1,GH2,Teo1,
TeoTof,Toft5,Toft6,Toft8B,Toft28,Toft15,Toft16,Toft15B}.

\par

In the last part of the paper we investigate
continuity of the entropy functional $E_\phi$
when acting on $M^p(\rr d)$ for $p\in [1,2]$
and $M^\Phi (\rr d)$ in \eqref{Eq:SpecOrlModSpIntro}.
More precisely, in Section \ref{sec3} we show that
$E_\phi$ is continuous on $M^\Phi (\rr d)$ and on
$M^p(\rr d)$ for $p\in [1,2)$, but fails to be continuous
on $M^2(\rr d)$. This might be surprising due to
the embeddings
$$
M^p(\rr d)\subseteq M^\Phi (\rr d) \underset{\text{dense}}
\subseteq M^2(\rr d), \qquad p<2,
$$
which shows that $M^\Phi (\rr d)$
in some sense is close to $M^2(\rr d)$.
See Theorem \ref{Thm:EntropyCont} and
Lemma \ref{Lem:OrlModCloseM2}
for details.

\par


\par

\section*{Acknowledgement}
Joachim Toft was supported by Vetenskapsr{\aa}det,
project number 2019-04890. Anupam Gumber and Nimit Rana thanks the Linnaeus 
university for providing excellent research facilities and kind hospitality
during their academic visit. Gumber was also supported by the Austrian
Science Fund (FWF) project P33217.

\par

\section{Preliminaries}\label{sec1}

\par

In the section we recall some basic facts on Gelfand-Shilov
spaces, Orlicz spaces, Orlicz modulation spaces,
pseudo-differential operators and Wigner distributions.
We also give some examples on Young functions,
Orlicz spaces and Orlicz modulation spaces. 
(See Examples \ref{Eq:SpecYoungFunc} and
\ref{Ex:OrlModSpace1}.) Notice that
Young functions are fundamental in the definition of Orlicz
spaces and Orlicz modulation spaces).

\par

\subsection{Gelfand-Shilov spaces}\label{subsec-Gelfand-Shilov}

\par

For a real number $s>0$, the (standard Fourier invariant) Gelfand-Shilov
space $\mathcal S_{s}(\rr d)$
($\Sigma _{s}(\rr d)$) of Roumieu type (Beurling type)
consists of all $f\in C^\infty (\rr d)$ such that
\begin{equation}\label{gfseminorm}
\nm f{\mathcal S_{s,h}}
\equiv
\sup_{\substack{\alpha, \beta \in \mathbf N^d \\ x\in \rr d}}
\frac {|x^\beta \partial ^\alpha
f(x)|}{h^{|\alpha  + \beta |}(\alpha ! \beta !)^s}
\end{equation}
is finite for some $h>0$ (for every $h>0$). We equip
$\mathcal S_{s}(\rr d)$ ($\Sigma _{s}(\rr d)$) by the canonical inductive limit
topology (projective limit topology) with respect to $h>0$, induced by
the semi-norms defined in \eqref{gfseminorm}.

\par

We have
\begin{equation}\label{GSembeddings}
\begin{aligned}
\maclS _s (\rr d) &\hookrightarrow \Sigma _t(\rr d)
\hookrightarrow \maclS _t (\rr d)
\hookrightarrow
\mascS (\rr d)
\\[1ex]
&\hookrightarrow \mascS '(\rr d) 
\hookrightarrow  \maclS _t'(\rr d)
\hookrightarrow  \Sigma _t'(\rr d) \hookrightarrow
\maclS _s '(\rr d),
\quad \frac 12\le s<t,
\end{aligned}
\end{equation}
with dense embeddings.
Here $A\hookrightarrow B$ means that
the topological space $A$ is continuously embedded in the 
topological space $B$. We also have
$$
\maclS _s(\rr d)=\Sigma _t(\rr d)= \{ 0\} ,\qquad
s<\frac 12,\ t\le \frac 12.
$$

\par

The \emph{Gelfand-Shilov distribution spaces} $\maclS _s'(\rr d)$ 
and $\Sigma _s'(\rr d)$, of Roumieu and Beurling 
types respectively, are the (strong) duals of $\mathcal S_s(\rr d)$ 
and $\Sigma _s(\rr d)$, respectively. It follows that if
$\mathcal S_{s,h}'(\rr d)$ is the $L^2$-dual of
$\mathcal S_{s,h}(\rr d)$ and $s\ge \frac 12$
($s > \frac 12$),
then $\mathcal S_s'(\rr d)$
($\Sigma _s'(\rr d)$) can be identified
with the projective limit (inductive limit) of
$\mathcal S_{s,h}'(\rr d)$ with respect to $h>0$. It follows that
\begin{equation}\label{Eq:GSspacecond2}
\mathcal S_s'(\rr d) = \bigcap _{h>0}\mathcal S_{s,h}'(\rr d)
\quad \text{and}\quad \Sigma _s'(\rr d) =\bigcup _{h>0}
\mathcal S_{s,h}'(\rr d)
\end{equation}
for such choices of $s$ and $\sigma$, see
\cite{GS,Pil1,Pil3} for details.


\par

We let the Fourier transform
$\mathscr F$ be given by
$$
(\mathscr Ff)(\xi )= \widehat f(\xi )
\equiv
(2\pi )^{-\frac d2}\int _{\rr
{d}} f(x)e^{-i\scal  x\xi }\, dx,
\quad \xi \in \rr d,
$$
when $f\in L^1(\rr d)$. Here $\scal \cdo \cdo$
denotes the usual
scalar product on $\rr d$.
The Fourier transform $\mathscr F$ extends
uniquely to homeomorphisms on $\mathscr S'(\rr d)$,
$\maclS _s'(\rr d)$ and on $\Sigma _s'(\rr d)$.
Furthermore,
$\mascF$ restricts to
homeomorphisms on $\mathscr S(\rr d)$,
$\maclS _s(\rr d)$ and on $\Sigma _s (\rr d)$,
and to a unitary operator on $L^2(\rr d)$.
Similar facts hold true
with partial Fourier transforms in place of
Fourier transform.

\par

Let $\phi \in \mascS  (\rr d)$ be fixed.
Then the \emph{short-time
Fourier transform} $V_\phi f$ of $f\in \mascS '
(\rr d)$ with respect to the \emph{window function} $\phi$ is
the tempered distribution on $\rr {2d}$, defined by
$$
V_\phi f(x,\xi )  =
\mascF (f \, \overline {\phi (\cdo -x)})(\xi ), \quad x,\xi \in \rr d.
$$
If $f ,\phi \in \mascS (\rr d)$, then it follows that
$$
V_\phi f(x,\xi ) = (2\pi )^{-\frac d2}\int _{\rr d} f(y)\overline {\phi
(y-x)}e^{-i\scal y\xi}\, dy, \quad x,\xi \in \rr d.
$$

\par

By \cite[Theorem 2.3]{Toft28} it follows that the definition of the map
$(f,\phi)\mapsto V_{\phi} f$ from $\mascS (\rr d) \times \mascS (\rr d)$
to $\mascS(\rr {2d})$ is uniquely extendable to a continuous map from
$\maclS _s'(\rr d)\times \maclS_s'(\rr d)$
to $\maclS_s'(\rr {2d})$, and restricts to a continuous map
from $\maclS _s (\rr d)\times \maclS _s (\rr d)$
to $\maclS _s(\rr {2d})$.
The same conclusion holds with $\Sigma _s$ in place of
$\maclS_s$, at each occurrence.

\par

In the following proposition we give characterizations of 
Gelfand-Shilov
spaces and their distribution spaces in terms of estimates
of the short-time Fourier transform.
We omit the proof since the first part follows from
\cite[Theorem 2.7]{GroZim}
and the second part from \cite[Proposition 2.2]{Toft18}.
See also \cite{CPRT10} for related results.
Here and in what follows, the notation
$A(\theta )\lesssim B(\theta )$, $\theta \in \Omega$,
means that there is a constant $c>0$ such that
$A(\theta )\le cB(\theta )$ holds
for all $\theta \in \Omega$. We also set
$A(\theta )\asymp B(\theta )$ when
$A(\theta )\lesssim B(\theta )\lesssim A(\theta )$.

\par

\begin{prop}\label{stftGelfand2}
Let $s\ge \frac 12$ ($s>\frac 12$), $\phi \in \maclS _s(\rr d)\setminus 0$
($\phi \in \Sigma _s(\rr d)\setminus 0$) and let $f$ be a
Gelfand-Shilov distribution on $\rr d$. Then the following is true:
\begin{enumerate}
\item $f\in \maclS _s (\rr d)$ ($f\in \Sigma_s(\rr d)$), if and only if
\begin{equation}\label{stftexpest2}
|V_\phi f(x,\xi )| \lesssim  e^{-r (|x|^{\frac 1s}+|\xi |^{\frac 1s})}, \quad x,\xi \in \rr d,
\end{equation}
for some $r > 0$ (for every $r>0$).
\item $f\in \maclS _s'(\rr d)$ ($f\in \Sigma _s'(\rr d)$), if and only if
\begin{equation}\label{stftexpest2Dist}
|V_\phi f(x,\xi )| \lesssim  e^{r(|x|^{\frac 1s}+|\xi |^{\frac 1s})}, \quad
x,\xi \in \rr d,
\end{equation}
for every $r > 0$ (for some $r > 0$).
\end{enumerate}
\end{prop}

\par

\subsection{Weight functions}\label{subsec1.2}

\par

A \emph{weight} or \emph{weight function} on $\rr d$ is a
positive function $\omega
\in  L^\infty _{loc}(\rr d)$ such that $1/\omega \in  L^\infty _{loc}(\rr d)$.
The weight $\omega$ is called \emph{moderate},
if there is a positive weight $v$ on $\rr d$ such that
\begin{equation}\label{moderate}
\omega (x+y) \lesssim \omega (x)v(y),\qquad x,y\in \rr d.
\end{equation}
If $\omega$ and $v$ are weights on $\rr d$ such that
\eqref{moderate} holds, then $\omega$ is also called
\emph{$v$-moderate}.
We note that \eqref{moderate}
implies that $\omega$ fulfills
the estimates
\begin{equation}\label{moderateconseq}
v(-x)^{-1}\lesssim \omega (x)\lesssim v(x),\quad x\in \rr d.
\end{equation}
We let $\mascP _E(\rr d)$ be the set of all moderate weights on $\rr d$.

\par

It can be proved that if $\omega \in \mascP _E(\rr d)$, then
$\omega$ is $v$-moderate for some $v(x) = e^{r|x|}$, provided the
positive constant $r$ is large enough (cf. \cite{Gro2.5}).
That is,
\eqref{moderate} implies
\begin{equation}\label{Eq:weight0}
\omega (x+y) \lesssim \omega(x) e^{r|y|}
\end{equation}
for some $r>0$. In particular, \eqref{moderateconseq} shows that
for any $\omega \in \mascP_E(\rr d)$, there is a constant $r>0$ such that
\begin{equation}\label{Eq:BoundWeights}
e^{-r|x|}\lesssim \omega (x)\lesssim e^{r|x|},\quad x\in \rr d.
\end{equation}

\par

We say that $v$ is
\emph{submultiplicative} if $v$ is even and
\eqref{moderate}
holds with $\omega =v$. In the sequel, $v$ and $v_j$ for
$j\ge 0$, always stand for submultiplicative weights if
nothing else is stated.

\par

We let $\mascP ^{0} _E(\rd)$ be the set of all $\omega\in \mascP _E(\rd)$
such that \eqref{Eq:weight0} holds for every $r>0$. We also let $\mascP (\rd)$
be the set of all $\omega\in \mascP _E(\rd)$ such that
$$
\omega (x+y) \lesssim \omega(x) (1+|y|)^r
$$
for some $r>0$.
Evidently,
$$
\mascP (\rd) \subseteq \mascP ^{0} _E(\rd) \subseteq \mascP _E(\rd).
$$

\par

\subsection{Orlicz Spaces}\label{subsec1.3}

\par

We recall that a function $\Phi:[0,\infty ] \to
[0,\infty ]$ is called \emph{convex} if
\begin{equation*}
\Phi(s_1 t_1+ s_2 t_2)
\leq s_1 \Phi(t_1)+s_2\Phi(t_2),
\end{equation*}
when
$s_j,t_j\in \mathbf{R}$
satisfy $s_j,t_j \ge 0$ and
$s_1 + s_2 = 1,\ j=1,2$.

\par

\begin{defn}\label{Def:YoungFunc}
A function $\Phi _0$ from $[0,\infty ]$ to
$[0,\infty ]$
is called a \emph{Young function} if
the following is true:
\begin{enumerate}
\item $\Phi _0$ is convex;

\vrum

\item $\Phi _0(0)=0$;

\vrum

\item $\lim
\limits _{t\to\infty} \Phi _0(t)=\Phi _0(\infty )=\infty$.
\end{enumerate}
A function $\Phi$ from $[0,\infty ]$ to
$[0,\infty ]$ is called a \emph{quasi-Young function}
(of order $p_0\in (0,1]$) if there is a Young function
$\Phi _0$ such that $\Phi (t)=\Phi _0(t^{p_0})$ when
$t\in [0,\infty ]$.
\end{defn}

\par

We observe that $\Phi _0$ and $\Phi$
in Definition \ref{Def:YoungFunc} might not
be continuous, because we permit
$\infty$ as function value. For example,
$$
\Phi (t)=
\begin{cases}
0,&\text{when}\ t \leq a
\\[1ex]
\infty ,&\text{when}\ t>a
\end{cases}
$$
is convex but discontinuous at $t=a$.

\par

It is clear that $\Phi _0$ and $\Phi$ in
Definition \ref{Def:YoungFunc} are
non-decreasing, because if $0\leq t_1\leq t_2$
and $s\in [0,1]$ is chosen such
that $t_1=st_2$ and $\Phi _0$ is the same as in
Definition \ref{Def:YoungFunc}, then
\begin{equation*}
    \Phi _0(t_1)=\Phi _0(st_2+(1-s)0)
    \leq s\Phi _0(t_2)+(1-s)\Phi _0(0)
    \leq \Phi _0(t_2),
\end{equation*}
since $\Phi _0(0) =0$ and $s\in [0,1]$. Hence every
(quasi-)Young function is increasing.

\par

\begin{defn}\label{Def:OrliczSpaces1}
Let $\Phi$ be a (quasi-)Young function and
let $\omega _0 \in \mascP _E(\rr d)$.
Then the Orlicz space
$L^{\Phi}_{(\omega_0)}(\rr d)$ consists
of all measurable functions
$f:\rr d \to \mathbf C$ such that
$$
\nm f{L^{\Phi}_{(\omega_0)}}
\equiv
\inf  \Sets{\lambda>0}
{\int_\Omega \Phi 
\left (
\frac{|f(x) \cdot \omega_0 (x)|}{\lambda}
\right )
\, dx\leq 1}
$$
is finite. Here $f$ and $g$ in $L^{\Phi}_{(\omega_0)}(\rr d)$
are equivalent if $f=g$ a.e.
\end{defn}

\par

We will also consider Orlicz spaces parameterized with
two (quasi-)Young functions.

\par

\begin{defn}\label{Def:OrliczSpaces2}
Let $\Phi _{j}$ be
(quasi-)Young functions, $j=1,2$ and let
$\omega \in \mascP _E (\rr {2d})$.
\begin{enumerate}
\item
The mixed Orlicz space ${L^{\Phi _1, \Phi _2}_{(\omega)}}
= {L^{\Phi _1, \Phi _2}_{(\omega)}}(\rr {2d})$ consists of all
measurable functions $f:\rr {2d} \to
\mathbf C$ such that
$$
\nm f{L^{\Phi _1, \Phi _2}_{(\omega)}} \equiv
\nm {f_{1,\omega}}{L^{\Phi _2}},
$$
is finite, where
$$
f_{1,\omega}(x_2)=\nm {f(\cdo ,x_2) \omega(\cdo, x_2)}
{L^{\Phi _{1}}}.
$$

\vrum

\item
The mixed Orlicz space ${L^{\Phi _1, \Phi _2}_{*,(\omega)}}
= {L^{\Phi _1, \Phi _2}_{*,(\omega)}}(\rr {2d})$
consists of all measurable functions $f:\rr {2d} \to
\mathbf C$ such that
$$
\nm f{L^{\Phi _1, \Phi _2}_{*,(\omega)}} \equiv
\nm {g}{L^{\Phi _2,\Phi _1}_{(\omega _0)}},
$$
is finite, where
$$
g(x,\xi )=f(\xi ,x),\ \omega _0(x,\xi )=\omega (\xi ,x).
$$
\end{enumerate}
\end{defn}

\par

In most of our situations we assume that $\Phi$
and $\Phi _j$ above are Young functions. A few
properties for Wigner distributions in Section
\ref{sec2} are deduced when $\Phi$ and $\Phi _j$
are allowed to be quasi-Young functions. The reader
who is not interested of such general results may
always assume that all quasi-Young functions
should be Young functions.

\par

It is well-known that if $\Phi$, $\Phi _1$ and $\Phi _2$
in Definitions \ref{Def:OrliczSpaces1}
and \ref{Def:OrliczSpaces2}
are Young functions, then the spaces
$L^{\Phi}_{(\omega _0)}(\rr d)$ and
$L^{\Phi _1,\Phi _2}_{(\omega )}(\rr {2d})$
are Banach spaces (see e.{\,}g.
Theorem 3 of
III.3.2 and Theorem 10 of III.3.3 in \cite{RaoRen1}).
If more generally, $\Phi$, $\Phi _1$ and $\Phi _2$
are quasi-Young functions of order $p_0\in (0,1]$,
then $L^{\Phi}_{(\omega _0)}(\rr d)$ and
$L^{\Phi _1,\Phi _2}_{(\omega )}(\rr {2d})$
are quasi-Banach spaces of order $p_0$. For the reader
who is not familiar with quasi-Banach spaces we here
give the definition.

\par

\begin{defn}
Let $\mascB$ be a vector space. Then the functional
$\nm \cdo{\mascB}$ on $\mascB$ is called a
quasi-norm of order $p_0\in (0,1]$, or an $p_0$-norm,
if the following conditions are fulfilled:
\begin{enumerate}
\item $\nm f{\mascB}\ge 0$ with equality only for $f=0$;

\vrum

\item $\nm {\alpha f}{\mascB}=|\alpha |\nm f{\mascB}$ for every
$\alpha \in \mathbf C$ and $f\in \mascB$;

\vrum

\item $\nm {f+g}{\mascB}^{p_0} \le
\nm f{\mascB}^{p_0}+\nm g{\mascB}^{p_0}$
for every $f,g\in \mascB$.
\end{enumerate}
The space $\mascB$ is called a quasi-Banach space
(of order $p_0$) or an $p_0$-Banach space, if $\mascB$ is complete
under the topology induced by the quasi-norm $\nm \cdo{\mascB}$.
\end{defn}

\par

We refer to \cite[Lemma 1.18]{TofUst}
for the proof of the following lemma.

\par

\begin{lemma}\label{T}
Let $\Phi , \Phi _j$ be quasi-Young functions, $j=1,2$,
$\omega_0, v_0 \in \mascP_E (\rd)$
and $\omega, v \in \mascP_E (\rdd)$ be such that $\omega_0$
is $v_0$-moderate and $\omega$ is $v$-moderate.
Then $L^{\Phi}_{(\omega_0)}(\rd)$ and
$L^{\Phi _{1}, \Phi _{2}}_{(\omega)}(\rr{2d})$ are
 invariant under translations, and
$$
\Vert f(\cdo - x)\Vert_{L^\Phi _{(\omega_0)}}
\lesssim
\Vert f\Vert_{L^\Phi _{(\omega_0)}} v_0(x),
\quad
f\in L^\Phi _{(\omega_0)}(\rd),\ x\in \rd\;,
$$
and
$$
\Vert f(\cdo - (x,\xi))\Vert_{L^{\Phi _{1}, \Phi _{2}}_{(\omega )}}
\lesssim
\Vert f\Vert_{L^{\Phi _{1}, \Phi _{2}}_{(\omega)}}v(x,\xi ),
\quad f\in L^{\Phi _{1}, \Phi _{2}}_{(\omega )} (\rdd ),\ (x,\xi ) \in \rdd.
$$
\end{lemma}

\par

In most situations we assume that the (quasi-)Young
functions should satisfy the $\Delta _2$-condition
(near origin), whose definition is recalled
as follows.

\par

\begin{defn}\label{Def:Delta2Cond}
Let $\Phi: [0,\infty ] \to [0,\infty]$ be a (quasi-)Young function.
Then $\Phi$ is said to satisfy the \emph{$\Delta_2$-condition}
if there exists a constant $C>0$ such that
\begin{equation}\label{Eq:Delta2Cond}
\Phi(2t) \leq C \Phi(t) 
\end{equation}
for every $t\in [0,\infty ]$. The (quasi-)Young function $\Phi$
is said to satisfy \emph{local $\Delta_2$-condition}
or \emph{$\Delta_2$-condition near origin}, if there are
constants $r>0$ and $C>0$ such that \eqref{Eq:Delta2Cond}
holds when $t\in [0,r]$.
\end{defn}

\par

\begin{rem}\label{Rem:Delta2Cond}
Suppose that $\Phi: [0,\infty ] \to [0,\infty]$
is a (quasi-)Young function which satisfies \eqref{Eq:Delta2Cond}
when $t\in [0,r]$ for some constants $r>0$ and $C>0$.
Then it follows by straight-forward arguments that
there is a quasi-Young function $\Phi _0$ (of the same order)
which satisfies the $\Delta _2$-condition
(on the whole $[0,\infty $), and such that
$\Phi _0(t)=\Phi (t)$ when $t\in [0,r]$).
\end{rem}

\par

Several duality properties for Orlicz spaces
can be described in terms of Orlicz spaces
with respect to Young conjugates, given in
the following definition.

\par

\begin{defn}\label{Def:ConjYoungFunc}
Let $\Phi$ be a Young function. Then
the conjugate Young function
$\Phi ^*$ is given
by
\begin{equation}\label{eq-YoungIneq-conjugate}
\Phi ^*(t)
\equiv
\begin{cases}
{\displaystyle{\sup _{s\ge 0} (st - \Phi(s)),}} &
\text{when}\ t \in [0,\infty ),
\\[2ex]
\infty , &
\text{when}\ t=\infty .
\end{cases}
\end{equation}
\end{defn}

\par

\begin{rem}\label{Rem:PhiLeb}
Let $p\in (0,\infty ]$, and set
$\Phi _{[p]}(t)= \frac{t^p}{p}$ when $p \in (0,\infty)$,
and
$$
\Phi _{[\infty ]} (t) =
\begin{cases}
0, & t \leq 1,
\\[1ex]
\infty ,&  t >1.
\end{cases}
$$
Then
$L^{\Phi _{[p]}}(\rr d)$ and its (quasi-)norm is equal to the
classical Lebesgue space $L^p(\rr d)$ and its (quasi-)norm. We
observe that $\Phi _{[p]}$ is a Banach space when $p\ge 1$
and a quasi-Banach space of order $p$ when $p\le 1$.
\end{rem}

\par

Due to the previous remark we observe that there are
Young functions which are not injective and
thereby fail to be invertible. In the following
definition we define some sort of pseudo-inverse
of such functions.

\par

\begin{defn}\label{Def:EssInv}
Let $\Phi : [0,\infty ]\to [0,\infty ]$
be a (quasi-)Young function and let
\begin{align*}
t_1 &= \sup \sets {t\ge 0}{\Phi (t)=0},
\\[1ex]
t_2 &= \sup \sets {t\ge 0}{\Phi (t)<\infty}
\intertext{and}
s_0 &= \sup \sets {\Phi (t)}{t<t_2}.
\end{align*}
Then $t_1$ is called the \emph{zero point} and
$t_2$ is called the \emph{infinity point} for $\Phi$,
and the \emph{essential inverse}
$\Phi ^{-\&}: [0,\infty ]\to [0,\infty ]$
for $\Phi$ is given by
$$
\Phi ^{-\&}(s)
=
\begin{cases}
0, & s=0,
\\[1ex]
t, & s=\Phi (t),\ t_1<t<t_2,
\\[1ex]
t_2, & s\ge s_0.
\end{cases}
$$
\end{defn}

\par

\begin{example}\label{Example:SpecYoungFunc}
We observe that if $t_1=0$ and $t_2=\infty$
in Definition \ref{Def:EssInv}, then $\Phi$
is invertible and $\Phi ^{-\&}$ agrees with
the inverse $\Phi ^{-1}$ of $\Phi$. For example,
for $\Phi _{[p]}$ with $p<\infty$
in Remark \ref{Rem:PhiLeb} we have
$$
\Phi _{[p]}^{-\&}(s) = \Phi _{[p]}^{-1}(s)
=
\begin{cases}
(ps)^{\frac 1p}, & 0\le s <\infty ,
\\[1ex]
\infty , & s=\infty .
\end{cases}
$$
For $p=\infty$, the inverse to $\Phi _{[\infty ]}$
does not exist, while the essential inverse
becomes
$$
\Phi _{[\infty ]} ^{-\&}(s) =
\begin{cases}
0, & s=0,
\\[1ex]
1, & s>0.
\end{cases}
$$

\par

Another example of a Young function is
$$
\Phi (t) =
\begin{cases}
\tan t, & 0\le t <\frac \pi 2,
\\[1ex]
\infty , & t\ge \frac \pi 2,
\end{cases}
$$
which also fails to be invertible. The
essential inverse becomes
$$
\Phi ^{-\&}(s) =
\begin{cases}
\arctan s, & 0\le s <\infty ,
\\[1ex]
\frac \pi 2, & s=\infty .
\end{cases}
$$

\par

We also observe that
$$
\Phi (t) =
\begin{cases}
0, & t=0,
\\[1ex]
-\frac t{\ln t}, & 0< t <1,
\\[1ex]
\infty , & t\ge 1,
\end{cases}
$$
is a Young function which does not satisfy the
$\Delta _2$-condition. Its essential inverse is
$$
\Phi ^{-\&}(s) =
\begin{cases}
\Phi ^{-1} (s), & 0\le s <\infty ,
\\[1ex]
1, & s=\infty .
\end{cases}
$$
We notice that the conjugate Young function
of $\Phi$ is given by
$$
\Phi ^*(t)
=
\left ( 
t+\frac 12 -\sqrt {\frac 14+t}\, 
\right )
e^{-\frac 1t(\frac 12+\sqrt {\frac 14+t}\, )},
$$
when $t\ge 0$ is near origin.

\par

We observe that each one of these Young functions
gives rise to different Orlicz spaces.
\end{example}

\par

We refer to \cite{SchF,RaoRen1,HaH} for more facts about
Orlicz spaces.

\par

\subsection{Orlicz modulation spaces}\label{subsec1.4}

\par

Before considering Orlicz modulation spaces, we recall the
definition of classical modulation spaces. (Cf. \cite{F1,Fei5}.)

\par

\begin{defn}\label{Def:Orliczmod}
Let $\phi(x) = \pi ^{-\frac{d}{4}}e^{-\frac{|x|^2}{2}},\ x\in \rd$,
$p,q\in (0,\infty]$, $\omega \in \mascP _E(\rr {2d})$, 
and let $\Phi$ and $\Psi$ be (quasi-)Young functions.
\begin{enumerate}
\item The \emph{modulation spaces} $M^{p,q}_{(\omega)}(\rr d)$
is set of all $f\in \maclS _{1/2}' (\rr d)$ such that
\begin{equation}\label{Eq:ClassicModNorm}
\nm f{M^{p,q}_{(\omega)}} \equiv \nm {V_\phi f}{L^{p,q}_{(\omega)}}
\end{equation}
is finite. The topology of $M^{p,q}_{(\omega)}(\rr d)$ is
given by the norm \eqref{Eq:ClassicModNorm}.

\vrum

\item The \emph{Orlicz modulation spaces}
$M^{\Phi}_{(\omega )} (\rr d)$,
$M^{\Phi , \Psi}_{(\omega )}(\rr {d})$
and
$W^{\Phi , \Psi}_{(\omega )}(\rr {d})$
are the sets of all $f\in \maclS _{1/2}' (\rr d)$ such that
\begin{equation}\label{Eq:OrlModNorm}
\nm f{M^{\Phi}_{(\omega )}}
\equiv
\nm {V_\phi f}{L^{\Phi}_{(\omega )}},
\quad
\nm f {M^{\Phi , \Psi}_{(\omega )}}
\equiv
\nm {V_\phi f} {L^{\Phi , \Psi }_{(\omega)}}
\quad \text{and} \quad
\nm f {W^{\Phi , \Psi}_{(\omega )}}
\equiv
\nm {V_\phi f} {L^{\Phi , \Psi }_{*,(\omega)}}
\end{equation}
respectively are finite. The topologies of
$M^{\Phi}_{(\omega )} (\rr d)$,
$M^{\Phi , \Psi}_{(\omega )}(\rr {d})$
and
$W^{\Phi , \Psi}_{(\omega )}(\rr {d})$
are given by the respective norms in
\eqref{Eq:OrlModNorm}.
\end{enumerate}
%
\end{defn}

\par

Let $\Phi$ and $\Psi$ be quasi-Young functions, and let
$\Phi _{[p]}$ be the same as
in Remark \ref{Rem:PhiLeb} and $\omega \in \mascP _E(\rr {2d})$.
Then evidently
\begin{alignat}{3}
M^{p,q}_{(\omega )}(\rr d)
&=
M^{\Phi ,\Psi}_{(\omega )}(\rr d) &
\quad &\text{when} & \quad
\Phi &=\Phi _{[p]},\ \Psi =\Phi _{[q]}.
\label{Eq:ModSpOrlModSp1}
\intertext{We now set}
M^{p,\Psi}_{(\omega )}(\rr d)
&=
M^{\Phi ,\Psi}_{(\omega )}(\rr d) &
\quad &\text{when} & \quad
\Phi &=\Phi _{[p]},
\label{Eq:ModSpOrlModSp2}
\intertext{and}
M^{\Phi ,q}_{(\omega )}(\rr d)
&=
M^{\Phi ,\Psi}_{(\omega )}(\rr d) &
\quad &\text{when} & \quad
\Psi &=\Phi _{[q]}.
\label{Eq:ModSpOrlModSp3}
\end{alignat}
For conveniency we also set
\begin{alignat*}{3}
M^{p,q}
&=
M^{p,q}_{(\omega )}, &
\quad
M^{p,\Psi}
&=
M^{p,\Psi}_{(\omega )}, &
\quad
M^{\Phi ,q}
&=
M^{\Phi ,q}_{(\omega )},
\\[1ex]
M^\Phi
&=
M^\Phi _{(\omega )}, &
\quad
M^{\Phi ,\Psi } &= M^{\Phi ,\Psi} _{(\omega )} &
\quad &\text{when}\quad
\omega (x,\xi)=1,
\end{alignat*}
and $M^p=M^{p,p}$ and $M^p_{(\omega)}=M^{p,p}_{(\omega)}$.

\par

Next we explain some basic properties of
Orlicz modulation spaces. The following
proposition shows that Orlicz modulation spaces
are completely determined by the behavior
of the quasi-Young functions near origin. We refer to
\cite[Proposition 5.11]{ToUsNaOz} for the proof. 


\par

\begin{prop}\label{Prop:OrliczModInvariance}
Let $\Phi _j$ and $\Psi _j$, $j=1,2$, be quasi-Young functions
and $\omega \in \mascP_E(\rdd )$.
Then the following conditions are equivalent:
\begin{enumerate}
\item $M^{\Phi _{1}, \Psi _1}_{(\omega )}(\rr d)\subseteq
M^{\Phi _{2} ,\Psi _{2}}_{(\omega )}(\rr d)$;

\vrum

\item for some $t_0>0$ it holds
$\Phi _{2} (t)\lesssim  \Phi _{1} (t)$
and
$\Psi _{2} (t)\lesssim  \Psi _{1} (t)$
when $t\in [0, t_0]$.
\end{enumerate}
\end{prop}

\par

The next two proposition show some other convenient
properties for Orlicz modulation spaces.

\par

\begin{prop}\label{Prop:BasicPropOrlModSp1}
Let $\Phi$, $\Phi _j$, $\Psi$, $\Psi _j$ be quasi-Young
functions, and let
$\omega ,\omega _j,v\in \mascP  _{E}(\rr {2d})$,
$j=1,2$, be such that $v$ is
submultiplicative and even, $\omega$ is $v$-moderate.
Then the following is true:
\begin{enumerate}
\item[{\rm{(1)}}] $M^{\Phi ,\Phi}_{(\omega )}(\rr d)
=
M^{\Phi}_{(\omega )}(\rr d)$, with equivalent quasi-norms;

\vrum

\item[{\rm{(2)}}] if $\phi \in \Sigma _1(\rr d)\setminus 0$,
then $f\in M^{\Phi ,\Psi}_{(\omega )}(\rr d)$,
if and only if $\nm {V_\phi f} {L^{\Phi , \Psi }_{(\omega)}}$ is finite.
Moreover, $M^{\Phi ,\Psi}_{(\omega )}(\rr d)$ is a
quasi-Banach space under the respective quasi-norm
in \eqref{Eq:OrlModNorm}, and different choices of
$\phi \in \Sigma _1(\rr d)\setminus 0$
give rise to equivalent norms. If more restrictive
$\Phi$ and $\Psi$ are Young functions, then
$M^{\Phi ,\Psi}_{(\omega )}(\rr d)$ is a Banach space, and
similar facts hold true with the condition
$\phi \in M^1_{(v)}(\rr d)\setminus 0$ in place of
$\phi \in \Sigma _1(\rr d)\setminus 0$ at each occurrence.

\vrum

\item[{\rm{(3)}}] if $\Phi _2\lesssim \Phi _1$, $\Psi _2\lesssim \Psi _1$
and $\omega _2\lesssim \omega _1$, then
$$
\Sigma _{1}(\rr d)
\subseteq
M^{\Phi _{1}, \Psi _1}_{(\omega _1)}(\rr d)
\subseteq
M^{\Phi _{2} ,\Psi _{2}}_{(\omega _2)}(\rr d)
\subseteq
\Sigma ' _{1}(\rr d)\text .
$$
\end{enumerate}
\end{prop}

\par

\par

\begin{prop}\label{Prop:BasicPropOrlModSp2}
Let $\Phi$, $\Psi$ be Young functions, and let
$\omega \in \mascP  _{E}(\rr {2d})$.
Then the following is true:
\begin{enumerate}
\item[{\rm{(1)}}] the sesqui-linear form $( \cdo ,\cdo )_{L^2}$ on
$\Sigma _1(\rr d)$ extends to a continuous map from
$$
M^{\Phi,\Psi}_{(\omega )}(\rr d)
\times
M^{\Phi ^*,\Psi ^*}_{(1/\omega )}(\rr d)
$$
to $\mathbf C$. This extension is unique when $\Phi$ and $\Psi$ fulfill
a local $\Delta _2$-condition. If $\nmm f = \sup |{(f,g)_{L^2}}|$, where
the supremum is
taken over all $b\in M^{\Phi ^*,\Psi ^*}_{(1/\omega )}(\rr d)$ such that
$\nm b{M^{\Phi ^*,\Psi ^*}_{(1/\omega )}}\le 1$, then
$\nmm {\cdot}$ and $\nm
\cdot {M^{p,q}_{(\omega )}}$ are equivalent norms;

\vrum

\item[{\rm{(2)}}] if $\Phi$ and $\Psi$ fulfill
a local $\Delta _2$-condition, then $\Sigma _1(\rr d)$ is dense
in $M^{\Phi,\Psi}_{(\omega )}(\rr d)$, and the dual space of
$M^{\Phi,\Psi}_{(\omega
)}(\rr d)$ can be identified with
$M^{\Phi ^*,\Psi ^*}_{(1/\omega )}(\rr
d)$, through the form $(\cdo  ,\cdo )_{L^2}$.
Moreover,
$\Sigma _1(\rr d)$ is weakly dense in
$M^{\Phi ^*,\Psi ^*}_{(\omega )}(\rr
d)$.
\end{enumerate}
\end{prop}

\par


Proposition \ref{Prop:BasicPropOrlModSp1}
follows from Theorem 2.4 in \cite{TofUst},
Theorems 3.1 and 5.9 in \cite{ToUsNaOz},
and Proposition
\ref{Prop:OrliczModInvariance}.
The details are left for the reader.
(See also Theorem 4.2 and other
results in \cite{FeiGro1}.)
Proposition \ref{Prop:BasicPropOrlModSp2}
is well-known in the case of modulation spaces
(see e.{\,}g. Chapters 11 and 12 in
\cite{Gro1}). For general Orlicz modulation
spaces, Proposition \ref{Prop:BasicPropOrlModSp2}
essentially follow from Propositions 4.3 and
4.9 in \cite{FeiGro1} and the
fact that similar results hold for
Orlicz spaces.

\par

In order to be self-contained we have
included a straight-forward proof of
Proposition \ref{Prop:BasicPropOrlModSp2}
in Appendix \ref{app:A},
with arguments adapted to the present
situation.


\par

\par

\begin{example}\label{Ex:OrlModSpace1}
Let $\Phi$ be a Young function given by
\eqref{Eq:SpecOrlModSpIntro}.
For the
entropy functional \eqref{Eq:EntropyFunc} it is
announced in the introduction that it
might be more suitable to investigate
such functional in background of
the Orlicz modulation space
$M^\Phi (\rr d)$ instead of
the classical modulation space or
Lebesgue space $M^2(\rr d)=L^2(\rr d)$
when $\Phi$ is given by \eqref{Eq:SpecOrlModSpIntro}.
(See \cite{MajLab1,MajLab2} and the references 
therein.) 

\par

In fact, in Section \ref{sec3} we show that
\begin{enumerate}
\item The functional $E$ is continuous on $M^\Phi (\rr d)$,
but fails to be continuous on $M^2(\rr d)$. (Cf. Theorem
\ref{Thm:EntropyCont}.)

\vrum

\item The space $M^\Phi (\rr d)$ is close to $M^2(\rr d)$ in
the sense of
\begin{equation*}
M^p(\rr d)\subseteq M^\Phi (\rr d) \underset{\text{dense}}
\subseteq M^2(\rr d),
\qquad p<2.
\end{equation*}
(Cf. Lemma \ref{Lem:OrlModCloseM2}.)
\end{enumerate}
\end{example}

\par

\subsection{Pseudo-differential operators}

\par

Next we recall some basic facts from
pseudo-differential calculus
(cf. \cite{Hor1}). Let $s\ge 1/2$, $a\in \maclS _s
(\rr {2d})$, and let $A$ belong to $\GL (d,\mathbf R)$,
the set of all $d\times d$-matrices with entries in $\mathbf R$.
Then the pseudo-differential operator $\op _A(a)$
defined by
\begin{equation}\label{e0.5}
\op _A(a)f(x) = (2\pi )^{-d}\iint _{\rr {2d}}
a(x-A(x-y),\xi )
f(y)e^{i\scal {x-y}\xi}\, dyd\xi
\end{equation}
is a linear and continuous operator on $\maclS _s (\rr d)$.
For $a\in \maclS _s'(\rr {2d})$ the
pseudo-differential operator $\op _A(a)$ is defined as
the continuous
operator from $\maclS _s(\rr d)$ to $\maclS _s'(\rr d)$ with
distribution kernel given by
\begin{equation}\label{atkernel}
K_{a,A}(x,y)=(2\pi )^{-\frac d2}(\mascF _2^{-1}a)(x-A(x-y),x-y).
\end{equation}
Here $\mascF _2F$ is the partial Fourier transform of $F(x,y)\in
\maclS _s'(\rr {2d})$ with respect to the variable $y \in \rr d$. This
definition generalizes \eqref{e0.5} and is well defined, since the mappings
\begin{equation}\label{homeoF2tmap}
\mascF _2\quad \text{and}\quad F(x,y)\mapsto F(x-A(x-y),y-x)
\end{equation}
are homeomorphisms on $\maclS _s'(\rr {2d})$.
The map $a\mapsto K_{a,A}$ is hence a homeomorphism on
$\maclS _s'(\rr {2d})$.

\par

If $A=0$, then $\op _A(a)$ is the standard or Kohn-Nirenberg representation
$a(x,D)$. If instead $A=\frac 12 I_d$, then $\op _A(a)$ agrees
with the Weyl operator or Weyl quantization $\op ^w(a)$. Here
$I_d$ is the $d\times d$ identity matrix.

\par

For any $K\in \maclS '_s(\rr {d_1+d_2})$, let $T_K$ be the
linear and continuous mapping from $\maclS _s(\rr {d_1})$
to $\maclS _s'(\rr {d_2})$ defined by
\begin{equation}\label{pre(A.1)}
(T_Kf,g)_{L^2(\rr {d_2})}
=
(K,g\otimes \overline f )_{L^2(\rr {d_1+d_2})},
\quad f \in \maclS _s(\rr {d_1}), \quad
g \in \maclS _s(\rr {d_2}).
\end{equation}
It is a well-known consequence of the Schwartz kernel
theorem that if $A\in \GL (d,\mathbf R)$, then
$K\mapsto T_K$ and $a\mapsto
\op _A(a)$ are bijective mappings from
$\mascS '(\rr {2d})$
to the space of linear and continuous mappings from
$\mascS (\rr d)$ to
$\mascS '(\rr d)$ (cf. e.{\,}g. \cite{Hor1}).

\par

Likewise the maps $K\mapsto T_K$
and $a\mapsto \op _A(a)$ are uniquely extendable to bijective
mappings from $\maclS _s'(\rr {2d})$ to the set of linear and
continuous mappings from $\maclS _s(\rr d)$
to $\maclS _s'(\rr d)$.
In fact, the asserted bijectivity for the map
$K\mapsto T_K$ follows from
the kernel theorems for topological vector spaces, using
the fact that Gelfand-Shilov spaces are inductive
or projective limits of certain Hilbert spaces
of Hermite series expansions (see \cite{Pil3,Kot1,SchWol}).
This kernel theorem corresponds to the
Schwartz kernel theorem in the usual distribution theory.
The other assertion follows from the fact that the map
$a\mapsto K_{a,A}$ is a homeomorphism on
$\maclS _s'(\rr {2d})$.

\par

In particular, for each $a_1\in \maclS _s '(\rr {2d})$ and $A_1,A_2\in
\GL (d,\mathbf R)$, there is a unique $a_2\in \maclS _s '(\rr {2d})$ such that
$\op _{A_1}(a_1) = \op _{A_2} (a_2)$. The relation between $a_1$ and $a_2$
is given by
\begin{equation}\label{calculitransform}
\op _{A_1}(a_1) = \op _{A_2}(a_2) \quad
\Leftrightarrow \quad
a_2(x,\xi )=e^{i\scal {(A_1-A_2)D_\xi }{D_x}}
a_1(x,\xi ).
\end{equation}
(Cf. \cite{Hor1}.) Note that the right-hand side
makes sense, since $\widehat a_2(\xi ,x)
=
e^{i\scal {(A_1-A_2)x}{\xi}}\widehat a_1(\xi ,x)$,
and that the map $a(\xi ,x)
\mapsto
e^{i\scal {Ax}\xi }a(\xi ,x)$ is continuous on
$\maclS _s '(\rr {2d})$ (see e.{\,}g. \cite{Tra}).

\par

The operator $e^{i\scal {AD_\xi }{D_x}}$ is
essential when transferring Wigner distributions
to each others. In what follows we have
the following continuity result for
$e^{i\scal {AD_\xi }{D_x}}$
when acting on Orlicz modulation spaces.

\par

\begin{prop}
Let $\Phi _1$ and $\Phi _2$ be quasi-Young functions,
$A\in \GL (d,\mathbf R)$, $s_1\ge \frac 12$,
$s_2>\frac 12$, $\omega _0\in \mascP _E(\rr {4d})$,
where $\mascP _E(\rr {4d})$ is  the set of all moderate
functions on $\rr {4d}$, and let
$$
\omega _A(x,\xi,\eta ,y)
=
\omega _0(x+Ay,\xi +A^*\eta ,\eta ,y).
$$
Then the following is true:
\begin{enumerate}
\item $e^{i\scal {AD_\xi}{D_x}}$ is continuous
from $\mascS (\rr {2d})$ to $\mascS '(\rr {2d})$,
and restricts to homoemorphisms on
$$
\maclS _{s_1}(\rr {2d}),
\quad
\Sigma _{s_2}(\rr {2d})
\quad \text{and}\quad
\mascS (\rr {2d}),
$$
and is uniquely extendable to homeomorphisms on
$$
\maclS _{s_1}'(\rr {2d}),
\quad
\Sigma _{s_2}'(\rr {2d})
\quad \text{and}\quad
\mascS '(\rr {2d}) \text ;
$$

\vrum

\item $e^{i\scal {AD_\xi}{D_x}}$ from
$\Sigma _1'(\rr {2d})$ to $\Sigma _1'(\rr {2d})$
restricts to a homeomorphism from
$M^{\Phi _1,\Phi _2}_{(\omega _0)}(\rr {2d})$ to
$M^{\Phi _1,\Phi _2}_{(\omega _A)}(\rr {2d})$, and
\begin{equation}\label{Eq:OscOpOrlModCont}
\nm {e^{i\scal {AD_\xi}{D_x}}a}
{M^{\Phi _1,\Phi _2}_{(\omega _A)}}
\asymp
\nm a{M^{\Phi _1,\Phi _2}_{(\omega _0)}},
\qquad
a\in M^{\Phi _1,\Phi _2}_{(\omega _0)}(\rr {2d}).
\end{equation}
\end{enumerate}
\end{prop}

\par

\begin{proof}
 We shall follow the proof of
\cite[Proposition 2.8]{Toft15B}.

\par

The assertion (1) and its proof
can be found in e.{\,}g.
\cite{CapTof,Tra}.
Let $T=e^{i\scal {AD_\xi}{D_x}}$.
By (2.12) in \cite{Toft15B} we have
$$
|(V_{T\phi}(Tf))(x,\xi ,\eta ,y)|
=
|(V_\phi f)(x+Ay,\xi +A^*\eta ,\eta ,y)|,
$$
when $\phi \in \Sigma _1(\rr {2d})$. 
By multiplying with $\omega _A$,
applying the $L^\Phi$
norm on the $x$ and $\xi$ variables and
then taking $x+Ay$ and $\xi +A^*\eta$
as new variables of integration give
$$
\nm {(V_{T\phi }(Tf))
(\cdo ,\eta ,y)\omega _A(\cdo ,\eta ,y)}
{L^\Phi (\rr {2d})}
=
\nm {(V_{\phi}f)
(\cdo ,\eta ,y)\omega (\cdo ,\eta ,y)}
{L^\Phi (\rr {2d})}.
$$
The relation \eqref{Eq:OscOpOrlModCont}
now follows by applying the
$L^\Psi$ norm with respect to the
$y$ and $\eta$ variables on the latter 
identity. This gives the result.
\end{proof}

\par

For future references we observe the relationship 
\begin{equation}\label{Eq:STFTLinkKernelSymbol}
\begin{aligned}
|(V_\phi K_{a,A})(x,y,\xi ,-\eta )|
&=
|(V_\psi a)(x-A(x-y),A^*\xi +(I-A^*)\eta ,\xi -\eta ,y-x)|,
\\[1ex]
\phi (x,y) &= (\mascF _2\psi )(x-A(x-y),x-y)
\end{aligned}
\end{equation}
between symbols and kernels for pseudo-differential 
operators, which follows by straight-forward
applications of Fourier inversion
formula (see also the proof of Proposition 2.5 in 
\cite{Toft15B}).

\par

\subsection{Wigner distributions}

\par

Next we recall general classes of Wigner distributions
parameterized by matrices. Let $A\in \GL (d,\mathbf R)$.
Then the $A$-Wigner distribution of
$f_1,f_2\in \mascS (\rr d)$, is defined by the formula
\begin{equation}\label{Eq:WignerDef}
W_{f_1,f_2}^A(x,\xi ) \equiv \mascF \big (f_1(x+A\cdo
)\overline{f_2(x+(A-I)\cdo )} \big ) (\xi ),
\end{equation}
which takes the form
$$
W_{f_1,f_2}^A(x,\xi ) =(2\pi )^{-\frac d2} \int
f_1(x+Ay)\overline{f_2(x+(A-I)y) }e^{-i\scal y\xi}\, dy,
$$
when $f_1,f_2\in \maclS _s (\rr d)$.
We set $W_{f_1,f_2}=W_{f_1,f_2}^A$ when
$A=\frac 12 I_d$ and $I_d$ is the $d\times d$,
in which case we get the classical Wigner
distribution.

\par

The definition of Wigner distributions
is extendable in various ways, which
the following result indicates. For the
proof we refer to \cite{Toft15B}
and its references.

\par

\begin{prop}
Let $A\in \GL (d,\mathbf R)$, $s\ge \frac 12$
and $T$ from
$\mascS (\rr d)\times \mascS (\rr d)$
to $\mascS '(\rr {2d})$ be the map given by
$(f_1,f_2)\mapsto W_{f_1,f_2}^A$. Then the
following is true:
\begin{enumerate}
\item $T$ restricts to a
continuous map from $\maclS _s (\rr d)\times
\maclS _s (\rr d)$
to $\maclS _s (\rr {2d})$, and is
uniquely extendable to a
continuous map from $\maclS _s '(\rr d)\times
\maclS _s '(\rr d)$
to $\maclS _s '(\rr {2d})$; 

\vrum

\item $T$ restricts to a
continuous map from $\maclS _s '(\rr d)\times
\maclS _s (\rr d)$ or from $\maclS _s (\rr d)\times
\maclS _s '(\rr d)$
to $\maclS _s '(\rr {2d})\cap C^\infty (\rr {2d})$.
\end{enumerate}

\par

The same holds true with $\mascS$ in place of
$\maclS _s$ at each occurrence. If in addition
$s>\frac 12$, then the same holds
true with $\Sigma _s$ in place of $\maclS _s$
at each occurrence.
\end{prop}

\par

The following result shows that Wigner
distributions with different matrices can
be carried over to each others. We refer
to Subsection 1.1 in \cite{Toft15B} for the
proof (see e.{\,}g. (1.10) in \cite{Toft15B}).

\par

\begin{lemma}\label{Lemma:WignerDistTransfers}
Let $A_1,A_2\in \GL(d,\mathbf R)$ and
$f_1,f_2\in \maclS _{1/2}'(\rr d)$. Then
$$
e^{i\scal {A_1D_\xi}{D_x}}W_{f_1,f_2}^{A_1}
=
e^{i\scal {A_2D_\xi}{D_x}}W_{f_1,f_2}^{A_2}.
$$
\end{lemma}

\par

Finally we recall the links
\begin{alignat}{3}
(\op _A(a)f,g)_{L^2(\rr d)}
&=
(2\pi )^{-\frac d2}
(a,W^A_{g,f})_{L^2(\rr {2d})}, &
\qquad a &\in \mascS '(\rr {2d}),&\ f,g &\in \mascS (\rr d),
\label{Eq:PseudoWignerLink}
\intertext{and}
\op _A(W_{f_1,f_2}^A)f(x) &= (2\pi )^{-\frac d2}(f,f_2)_{L^2}f_1(x), &
\qquad f_1,f_2 &\in \mascS '(\rr {d}),&\ f &\in \mascS (\rr d),
\label{Eq:PseudoWignerLink2}
\end{alignat}
between pseudo-differential operators
and Wigner distributions, which follows
by straight-forward computations. Similar facts hold true with
$\maclS _s$ or $\Sigma _s$ in place of $\mascS$ at each
occurrence.

\par

\begin{rem}
We observe that the definition of Wigner distributions
can be extended in various ways. For example,
metaplectic Wigner distributions are given in \cite{CorGia}.
\end{rem}

\par

\section{Continuity for pseudo-differential
operators when acting on Orlicz
modulation spaces}\label{sec2}

\par

In this section we deduce continuity properties for
Wigner distributions when acting on Orlicz modulation
spaces. Thereafter we apply such results to obtain
continuity properties for pseudo-differential operators
with symbols in Orlicz modulation spaces when acting
on other Orlicz modulation spaces.

\par

We need the following result on H{\"o}lder inequality for 
Orlicz spaces, and refer to \cite[III.3.3]{RaoRen1} for the
proof (cf. Theorem 7 in \cite[III.3.3]{RaoRen1}).
Here we let $S(\mu )$
be the set of all simple and ($\mu$-)measurable
functions on the measurable space $(E,\mu )$.

\par

\begin{prop}[H\"{o}lder inequality]\label{thm-Holder}
Let $(E,\mu )$ be a measurable space, and
$\Phi_j$, $j=0,1,2$ be Young's functions such that 
\begin{alignat*}{2}
\Phi_0(t_1t_2) &\leq \Phi_1(t_1) + \Phi_2(t_2),
& \qquad t_1,t_2 &\geq 0,
\intertext{or}
\Phi _0^{-1}(s)
&\ge
\Phi _1^{-1}(s)\cdot \Phi_2^{-1}(s), &
\qquad s &\geq 0.
\end{alignat*}
Then the map $(f_1,f_2)\mapsto f_1\cdot f_2$
from $S(\mu)\times S(\mu )$ to $S(\mu)$ extends
uniquely to a continuous map from
$L^{\Phi _1}(\mu)\times L^{\Phi _2}(\mu)$
to $L^{\Phi _0}(\mu)$,
and 
\begin{equation}\label{eq-Holder}
    \Vert f_1 \cdot f_2 \Vert_{L^{\Phi_0}}
    \leq
    2 \Vert f_1 \Vert_{L^{\Phi_1}} \Vert  f_2 \Vert_{L^{\Phi_2}},
    \qquad
    f_j\in L^{\Phi _j}(\mu),\ j=1,2.
\end{equation}
\end{prop}

\par

\subsection{Continuity for Wigner distributions
and short-time Fourier transforms on Orlicz
modulation spaces}

\par

It is natural to assume that our (quasi-)Young
functions should obey conditions given
in the following definition.

\par

\begin{defn}\label{Def:SteeredFunc}
Let $\Phi : [0,\infty ]\to [0,\infty ]$
and $p\in (0,\infty)$. Then $\Phi$
is called \emph{$p$-steered} if
one of the following conditions
are fulfilled:
\begin{enumerate}
\item $\limsup _{t\to 0+}
\frac {\Phi (t)}{t^p}=\infty$;

\vrum

\item $t\mapsto \Phi (t^{\frac 1p})$
is equal to a Young function near
origin.
\end{enumerate}
\end{defn}

\par

The first main result of the section is the following 
theorem which concerns continuity property for
Wigner distributions acting on
Orlicz modulation spaces. Here the involved
weight functions should satisfy
\begin{equation}\label{Eq:WeightCond1}
\omega (x,\xi ,\eta ,y)
\lesssim
\omega _1(x-Ay,\xi +(I-A^*)\eta )
\omega _2(x+(I-A)y,\xi -A^*\eta ).
\end{equation}

\par

\begin{thm}\label{Thm:WignerDistOrlicModSpace}
Let $A\in \GL (d,\mathbf R)$, $p,q\in [1,\infty ]$
be such that $p\le q$,
and let
$\Phi _j,\Psi _j : [0,\infty ]\to [0,\infty ]$, $j=1,2$,
be such that the following is true:
\begin{itemize}
\item if $p=\infty$, then $\Phi _j$ and $\Psi _j$ are Young
functions;

\vrum

\item if $p<\infty$, then
$\Phi _j$ and $\Psi _j$ are $p$-steered
Young functions
which fulfill a local $\Delta _2$-condition,
and for some $r>0$, it holds
\begin{alignat}{3}
\Phi _1(t),\Phi _2(t) &\gtrsim t^q & \quad
\Psi _1(t),\Psi _2(t) &\gtrsim t^q,
& \quad t &\in [0,r],
\label{Eq:WignerDistYoungFuncCond1}
\intertext{and}
\Phi _1^{-\&}(s)\Phi _2^{-\&}(s)
&\lesssim
s^{\frac 1p+\frac 1q}, &
\quad 
\Psi _1^{-\&}(t)\Psi _2^{-\&}(s)
&\lesssim
s^{\frac 1p+\frac 1q},
&\quad s &\in [0,r].
\label{Eq:WignerDistYoungFuncCond2}
\end{alignat}
\end{itemize}
Also let $\omega \in \mascP _E(\rr {4d})$ and
$\omega _1,\omega _2\in \mascP _E(\rr {2d})$
be such that \eqref{Eq:WeightCond1} holds.
Then the map $(f_1,f_2)\mapsto W_{f_1,f_2}^A$
from $\Sigma _1'(\rr d)
\times \Sigma _1'(\rr d)$ to $\Sigma _1'(\rr {2d})$
restricts to a continuous
map from $M^{\Phi _1,\Psi _1}_{(\omega _1)}(\rr d)
\times M^{\Phi _2,\Psi _2}_{(\omega _2)}(\rr d)$ to
$M^{p,q}_{(\omega )}(\rr {2d})$, and
\begin{equation}
\nm {W_{f_1,f_2}^A}{M^{p,q}_{(\omega )}}
\lesssim
\nm {f_1}{M^{\Phi _1,\Psi _1}_{(\omega _1)}}
\nm {f_2}{M^{\Phi _2,\Psi _2}_{(\omega _2)}},
\quad f_j\in M^{\Phi _j,\Psi _j}_{(\omega _j)}(\rr d),
\ j=1,2.
\end{equation}
\end{thm}

\par

\begin{rem}\label{Rem:PowerArgumOrlicz}
Suppose that $p,q\in (0,\infty ]$
satisfy $p\le q$, similarly as in
Theorem \ref{Thm:WignerDistOrlicModSpace},
let $\Phi :[0,\infty ] \to [0,\infty ]$
and let $\Phi _{[q]}(t)=\Phi (t^{\frac 1q})$.
Then the following is true:
\begin{enumerate}
\item if $q<\infty$ and
$\Phi _{[q]}$
is a Young function, then
$\Phi (t)\lesssim t^q$ near origin
and $\Phi _{[p]}$ is a Young function;

\vrum

\item if $\Phi _{[q]}$ is a Young function
which satisfies the $\Delta _2$-condition,
then $\Phi _{[p]}$ is a Young function
which satisfies the $\Delta _2$-condition.

\vrum

\item Suppose that $p<\infty$, and
$\Phi _j,\Psi _j: [0,\infty ]\to [0,\infty ]$
are such that $t\mapsto \Phi _j(t^{\frac 1p})$
and $t\mapsto \Psi _j(t^{\frac 1p})$ are Young
functions, $j=1,2$, and that
\eqref{Eq:WignerDistYoungFuncCond2} holds.
Then \eqref{Eq:WignerDistYoungFuncCond1}
holds. In particular, for some
$r_1,r_2>0$ it holds
\begin{alignat}{2}
t^q
&\lesssim
\Phi _j(t),\Psi _j(t)
\lesssim
t^p, &
\quad
t&\in [0,r_1],
\intertext{or equivalently,}
s^{\frac 1p}
&\lesssim
\Phi _j^{-\&}(s),\Psi _j^{-\&}(s)
\lesssim
s^{\frac 1q}, &
\quad
s&\in [0,r_2].
\end{alignat}
\end{enumerate}
\end{rem}

\par

For $q=\infty$ we observe the following consequence
of Theorem \ref{Thm:WignerDistOrlicModSpace},
for extreme choices of $\Phi _j$ or $\Psi _j$,
for some $j=1,2$.

\par

\begin{cor}\label{Cor:WignerDistOrlicModSpaceLargeq}
Let $A\in \GL (d,\mathbf R)$, $p\in [1,\infty )$
be such that $p\le q$,
and let $\Phi _j$ and $\Psi _j$ be such that
$t\mapsto \Phi _j(t^{\frac 1p})$ and
$t\mapsto \Psi _j(t^{\frac 1p})$ are Young functions 
which fulfill the $\Delta _2$-condition, $j=1,2$,
and such that
\begin{equation}\label{Eq:WignerDistYoungFuncCondLargeq}
\Phi _1^{-1}(s)\Phi _2^{-1}(s)
\le
s^{\frac 1p}
\quad \text{and}\quad
\Psi _1^{-1}(s)\Psi _2^{-1}(s)
\le
s^{\frac 1p}.
\end{equation}
Also let $\omega \in \mascP _E(\rr {4d})$ and
$\omega _1,\omega _2\in \mascP _E(\rr {2d})$
be such that \eqref{Eq:WeightCond1} holds.
Then the map $(f_1,f_2)\mapsto W_{f_1,f_2}^A$
from $\Sigma _1'(\rr d)
\times \Sigma _1'(\rr d)$ to $\Sigma _1'(\rr {2d})$
restricts to a continuous
map from 
\begin{alignat*}{2}
M^{p,\Psi _1}_{(\omega _1)}(\rr d)
&\times
M^{\infty ,\Psi _2}_{(\omega _2)}(\rr d) ,&
\quad
M^{\infty,\Psi _1}_{(\omega _1)}(\rr d)
&\times
M^{p,\Psi _2}_{(\omega _2)}(\rr d) ,
\\[1ex]
M^{\Phi _1,p}_{(\omega _1)}(\rr d)
&\times
M^{\Phi _2,\infty}_{(\omega _2)}(\rr d) &
\quad \text{or}\quad
M^{\Phi _1,\infty}_{(\omega _1)}(\rr d)
&\times
M^{\Phi _2,p}_{(\omega _2)}(\rr d) ,
\end{alignat*}
to $M^{p,\infty}_{(\omega )}(\rr {2d})$.
\end{cor}

\par


We need the following Young type results
for Orlicz spaces for the proof of Theorem
\ref{Thm:WignerDistOrlicModSpace}
(see Theorem 9 in \cite[III.3.3]{RaoRen1}).

\par

\begin{lemma}\label{Lemma:YoungIneq}
Let $\Phi _j,j=0,1,2$, be Young functions 
which fulfill the $\Delta _2$-condition
and such that 
\begin{equation}\label{eq-Youngfuncs-equality}
 \Phi _1^{-1}(s) \cdot \Phi _2^{-1}(s)
\leq 
s\Phi _0^{-1}(s), \qquad s \geq 0.   
\end{equation}
Then the convolution map $(f_1,f_2)\mapsto f_1*f_2$
from
$L^{\Phi _1}(\rr d)\times L^{\Phi _2}(\rr d)$
to $L^{\Phi _0}(\rr d)$ is continuous and
\begin{equation}\label{eq-YoungIneq}
\nm {f_1*f_2}{L^{\Phi _0}}
\le 2 \nm {f_1}{L^{\Phi _1}}
\nm {f_2}{L^{\Phi _2}},
\quad \forall f_j\in L^{\Phi _j}(\rr d),\ j=1,2.
\end{equation}
\end{lemma}


\par


\par

\par

\begin{proof}[Proof of Theorem 
\ref{Thm:WignerDistOrlicModSpace}]
First suppose $p=\infty$. Then it follows from e.{\,}g.
\cite[Proposition 2.4]{Toft15B}
that $(f_1,f_2)\mapsto W_{f_1,f_2}^A$
is continuous from
$M^\infty _{(\omega _1)}(\rr d)
\times
M^\infty _{(\omega _1)}(\rr d)$ to
$M^\infty _{(\omega )}(\rr {2d})$. The result
now follows from the facts that $q\ge p=\infty$
and
$M^{\Phi _j,\Psi _j}_{(\omega _j)}(\rr d)$
are continuously embedded in
$M^\infty _{(\omega _j)}(\rr d)$, $j=1,2$,
in view of Proposition
\ref{Prop:OrliczModInvariance}.

\par

It remains to consider the case when
$p<\infty$.
Since $M^{\Phi _j,\Psi _j}_{(\omega _j)}(\rr d)$
only depends on $\Phi _j$ and $\Psi _j$ near origin, 
in view of \cite[Proposition 5.9]{ToUsNaOz}, we
may replace these Young functions with new ones such 
that \eqref{Eq:WignerDistYoungFuncCond2}
holds for all $s\in [0,\infty ]$ (see
Definition \ref{Def:EssInv}).
Furthermore, by (1.10) and Lemma 2.6 in 
\cite{Toft15B}, it follows that we may
reduce ourselves to the case when
$A=\frac 12 I$, giving the standard
(cross-)Wigner distribution.

\par

First we consider the case when
$t\mapsto \Phi _j(t^{\frac 1p})$
and $t\mapsto \Psi _j(t^{\frac 1p})$
are Young functions. 
As a first step on this we also assume
that $q<\infty$, giving that
$p<\infty$.

\par

Let
$$
F=W_{f_1,f_2}
\quad \text{and}\quad
\psi =W_{\phi _1,\phi _2}.
$$
Then \cite[Lemma 2.6]{Toft15B} gives
$$
|V_\psi F(x,\xi ,\eta ,y)|
=
|V_{\phi _1}f_1(x-{\textstyle{\frac 12}}y,
\xi +{\textstyle{\frac 12}}\eta )|\cdot
|V_{\phi _2}f_2(x+{\textstyle{\frac 12}}y,
\xi -{\textstyle{\frac 12}}\eta )|.
$$
Hence, if
\begin{align*}
G(x,\xi ,\eta ,y)
&=
|V_\Phi F(x,\xi ,\eta ,y)\omega (x,\xi ,\eta ,y)|,
\\[1ex]
G_1(x,\xi )
&=
|V_{\phi _1} f_1(-x,\xi )\omega _1(-x,\xi )|
\intertext{and}
G_2(x,\xi )
&=
|V_{\phi _2} f_2(x,-\xi )\omega _2(x,-\xi )|
\end{align*}
then it follows from the assumptions that
\begin{equation}\label{Eq:GFunctionsRel}
0\le G(x,\xi ,\eta ,y)
\lesssim
G_1({\textstyle{\frac 12}}y-x,
{\textstyle{\frac 12}}\eta +\xi )
\cdot
G_2({\textstyle{\frac 12}}y+x,
{\textstyle{\frac 12}}\eta -\xi ).
\end{equation}

\par

By first applying the $L^p$-norm on the
$x$ and $\xi$ variables, and then
the $L^q$ norm on the $y$ variable we obtain
\begin{equation}\label{Eq:STFTBoundByR}
\nm {H_0(\eta ,\cdo )}{L^q(\rr d)}
\lesssim R(\eta ),
\qquad
H_0 (\eta ,y)=\nm {G(\cdo ,\eta ,y)}{L^p(\rr {2d})},
%
\end{equation}
where
$$
R(\eta )
\equiv
\left ( \int \left ( \iint
G_1({\textstyle{\frac 12}}y-x,
{\textstyle{\frac 12}}\eta +\xi )^p
G_2({\textstyle{\frac 12}}y+x,
{\textstyle{\frac 12}}\eta -\xi )^p
\, dxd\xi 
\right )^{\frac qp}\, dy \right )^{\frac 1q}.
$$
We need to estimate $R(\eta )$ in suitable ways.

\par

By taking $x+\frac 12 y$, $\xi -\frac 12\eta$
and $y$ as new variables of integrations,
and using Minkowski's inequality, we obtain 
\begin{align}
R(\eta )
&=
\left ( \int \left ( \iint
G_1(y-x,\eta -\xi )^p
G_2(x,\xi )^p
\, dxd\xi 
\right )^{\frac qp}\, dy \right )^{\frac 1q}
\notag
\\[1ex]
&\le
\left ( \int \left ( \int \left ( \int
G_1(y-x,\eta -\xi )^p
G_2(x,\xi )^p
\, dx \right )^{\frac qp}\, dy \right )^{\frac pq}
\, d\xi  \right )^{\frac 1p}
\notag
\\[1ex]
&=
\left ( \int \left ( 
\nm {G_1(\cdo ,\eta -\xi )^p*
G_2(\cdo ,\xi )^p}{L^{\frac qp}}\right )
\, d\xi  \right )^{\frac 1p}.
\label{Eq:MinkowskiSTFTEst}
\end{align}

\par

Now recall that if
$$
\widetilde \Phi _j(t)
=\Phi _j(t^{\frac 1p})
\quad \text{and}\quad
\widetilde \Psi _j(t)
=\Psi _j(t^{\frac 1p}),
\quad
j=1,2,
$$
then, since $\Phi _j$ and $\Psi _j$
for $j=1,2$ are $p$-steered, 
$\widetilde \Phi _j(t)$ and
$\widetilde \Psi _j(t)$ are Young functions
such that
$$
\widetilde \Phi _1^{-1}(s)
\widetilde \Phi _2^{-1}(s)
\le s^{\frac pq+1}
\quad \text{and}\quad
\widetilde \Psi _1^{-1}(s)
\widetilde \Psi _2^{-1}(s)
\le s^{\frac pq+1}.
$$
Hence Lemma \ref{Lemma:YoungIneq} gives
\begin{equation}\label{Eq:ConvSTFTIneq}
\nm {G_1(\cdo ,\eta -\xi )^p*
G_2(\cdo ,\xi )^p}{L^{\frac qp}}
\lesssim
H_1(\eta -\xi )^pH_2(\xi )^p,
\end{equation}
where
\begin{equation}\label{Eq:HjDef}
H_j(\xi )
\equiv
\nm {G_j(\cdo ,\xi )^p}
{L^{\widetilde \Phi _1}}^{\frac 1p},\quad
j=1,2 .
\end{equation}

\par

By combining \eqref{Eq:MinkowskiSTFTEst} with 
\eqref{Eq:ConvSTFTIneq} we obtain
$$
R(\eta )
\lesssim
((H_1^p*H_2^p)(\eta ))^{\frac 1p}.
$$
By applying the $L^q$ norm, and using that
$\nm F{M^{p,q}_{(\omega )}}\lesssim \nm R{L^q}$, due to
\eqref{Eq:STFTBoundByR} and Lemma
\ref{Lemma:YoungIneq} we get
\begin{align*}
\nm F{M^{p,q}_{(\omega )}}
&\lesssim
\nm R{L^q}
\lesssim
\nm {H_1^p*H_2^p}{L^{q/p}}^{\frac 1p}
\lesssim
\left (\nm {H_1^p}{L^{\widetilde \Psi _1}}
\nm {H_2^p}{L^{\widetilde \Psi _2}}\right )^{\frac 1p}
\\[1ex]
&=
\left (\nm {G_1^p}
{L^{\widetilde \Phi _1,\widetilde \Psi _1}}
\nm {G_2^p}
{L^{\widetilde \Phi _2,\widetilde \Psi _2}}
\right )^{\frac 1p}
= 
\nm {G_1}
{L^{\Phi _1,\Psi _1}}
\nm {G_2}
{L^{\Phi _2,\Psi _2}}
\asymp
\nm {f_1}{M^{\Phi _1,\Psi _1}_{(\omega _1)}}
\nm {f_2}{M^{\Phi _2,\Psi _2}_{(\omega _2)}},
\end{align*}
giving the result in the case $q<\infty$.

\par

Next suppose that $q=\infty$.
By first applying the $L^p$-norm on the
$x$ and $\xi$ variables, and then
the $L^\infty$ norm
on the $y$ variable in
\eqref{Eq:GFunctionsRel} we obtain
\begin{equation}\tag*{(\ref{Eq:STFTBoundByR})$'$}
\sup _{y\in \rr d}
\left ( \nm {G(\cdo ,\eta ,y)}{L^p(\rr {2d})}
\right )
\lesssim R(\eta ),
\end{equation}
where $R(\eta )$ is now redefined as
$$
R(\eta )
\equiv
\sup _{y\in \rr d}
\left ( \iint
G_1({\textstyle{\frac 12}}y-x,
\xi +{\textstyle{\frac 12}}\eta )^p
G_2(x+{\textstyle{\frac 12}}y,
{\textstyle{\frac 12}}\eta -\xi )^p
\, dxd\xi 
\right )^{\frac 1p}.
$$
By \eqref{Eq:MinkowskiSTFTEst} we have
\begin{align}
R(\eta )
&\le
\left ( \int 
\nm {G_1(\cdo ,\eta -\xi )^p*
G_2(\cdo ,\xi )^p}{L^\infty}
\, d\xi  \right )^{\frac 1p}
\tag*{(\ref{Eq:MinkowskiSTFTEst})$'$}
\end{align}

\par

We have
$$
\widetilde \Psi _1^{-1}(s)\widetilde \Psi _2^{-1}(s)
\le s,
$$
and by H{\"o}lder's inequality we obtain
\begin{equation}\tag*{(\ref{Eq:ConvSTFTIneq})$'$}
\nm {G_1(\cdo ,\eta -\xi )^p*
G_2(\cdo ,\xi )^p}{L^\infty}
\lesssim
H_1(\eta -\xi )^pH_2(\xi )^p,
\end{equation}
where $H_j$ are the same as in \eqref{Eq:HjDef}.

\par

A combination of \eqref{Eq:MinkowskiSTFTEst}$'$ and
\eqref{Eq:ConvSTFTIneq}$'$ gives
$$
R(\eta )
\lesssim
(H_1^p*H_2^p)(\eta )^{\frac 1p}.
$$
By applying the $L^\infty$ norm, and using that
$\nm F{M^{p,\infty}_{(\omega )}}\lesssim \nm R{L^\infty}$,
due to \eqref{Eq:STFTBoundByR}$'$, we obtain
\begin{align*}
\nm F{M^{p,\infty}_{(\omega )}}
&\lesssim
\nm R{L^\infty}
\lesssim
\nm {H_1^p*H_2^p}{L^\infty}
\\[1ex]
&\lesssim
\left (\nm {H_1^p}{L^{\widetilde \Psi _1}}
\nm {H_2^p}{L^{\widetilde \Psi _2}}\right )^{\frac 1p}
=
\left (\nm {G_1^p}
{L^{\widetilde \Phi _1,\widetilde \Psi _1}}
\nm {G_2^p}
{L^{\widetilde \Phi _2,\widetilde \Psi _2}}
\right )^{\frac 1p}
\\[1ex]
&=
\nm {G_1}
{L^{\Phi _1,\Psi _1}}
\nm {G_2}
{L^{\Phi _2,\Psi _2}}
\asymp
\nm {f_1}{M^{\Phi _1,\Psi _1}_{(\omega _1)}}
\nm {f_2}{M^{\Phi _2,\Psi _2}_{(\omega _2)}},
\end{align*}
giving the result in the case when
$t\mapsto \Phi _j(t^{\frac 1p})$
and $t\mapsto \Psi _j(t^{\frac 1p})$
are Young functions.

\par

It remains to consider the case when
$t\mapsto \Phi _j(t^{\frac 1p})$ or
$t\mapsto \Psi _k(t^{\frac 1p})$ are
not Young functions
for some $j=1,2$ and some $k=1,2$.
We shall here mainly use similar arguments
as in the proof of Theorem 1.1 in
\cite{CorNic}.
Then
$$
\Phi _j^{-\&}(s)\lesssim s^{\frac 1p}
\quad \text{or}\quad
\Psi _k^{-\&}(s)\lesssim s^{\frac 1p}
$$
near origin,
for some $j=1,2$ and some $k=1,2$.
First suppose that $\Phi _1^{-\&}(s)\lesssim s^{\frac 1p}$,
and that $\Psi _j(t^{\frac 1p})$ are Young
functions, $j=1,2$. Then
$$
\Phi _1^{-\&}(s)s^{\frac 1q}
\lesssim
s^{\frac 1p+\frac 1q},
$$
and
\begin{equation}\label{Eq:EmbOrlModInProof}
M^{\Phi _1,\Psi _1}_{(\omega _1)}(\rr d)
\subseteq
M^{p,\Psi _1}_{(\omega _1)}(\rr d)
\quad \text{and}\quad
M^{\Phi _2,\Psi _2}_{(\omega _2)}(\rr d)
\subseteq
M^{q,\Psi _2}_{(\omega _2)}(\rr d),
\end{equation}
where the last embedding follows from
\eqref {Eq:WignerDistYoungFuncCond1}.
By the previous part of the proof
we have that $(f_1,f_2)\mapsto W_{f_1,f_2}^A$
is continuous from
$M^{p,\Psi _1}_{(\omega _1)}(\rr d)
\times
M^{q,\Psi _2}_{(\omega _2)}(\rr d)$
to $M^{p,q}_{(\omega )}(\rr {2d})$.
The result now follows in this case
by combining the latter continuity
property with the embeddings in
\eqref{Eq:EmbOrlModInProof}.

\par

By similar arguments, the same conclusion holds 
true if instead
$$
\Phi _2^{-\&}(s)\lesssim s^{\frac 1p},
\quad
\Psi _1^{-\&}(s)\lesssim s^{\frac 1p}
\quad \text{or}\quad
\Psi _2^{-\&}(s)\lesssim s^{\frac 1p}.
$$
The details are left for the reader.

\par

Finally suppose that
$$
\Phi _j^{-\&}(s)\lesssim s^{\frac 1p},
\quad
\quad \text{and}\quad
\Psi _k^{-\&}(s)\lesssim s^{\frac 1p},
$$
for some $j=1,2$ and some $k=1,2$.
Then the previous arguments lead to
\begin{equation}\label{Eq:EmbOrlModInProof2}
M^{\Phi _j,\Psi _j}_{(\omega _j)}(\rr d)
\subseteq
M^{p_j,q_j}_{(\omega _j)}(\rr d),
\end{equation}
for some $p_j,q_j\in \{ p,q\}$,
$j=1,2$, and such that $p_1\neq p_2$
and $q_1=q_2$. Again we have that
$(f_1,f_2)\mapsto W_{f_1,f_2}^A$
is continuous from
$M^{p_1,q_1}_{(\omega _1)}(\rr d)
\times
M^{p_2,q_2}_{(\omega _2)}(\rr d)$
to $M^{p,q}_{(\omega )}(\rr {2d})$.
The asserted continuity now follows
from \eqref{Eq:EmbOrlModInProof2},
and the result follows.
\end{proof}

\par

Beside the estimates for Wigner distributions
on Orlicz modulation spaces in Theorem
\ref{Thm:WignerDistOrlicModSpace}, we also have
the following result on estimates for
the short-time Fourier transform.
The result generalizes 
\cite[Proposition 3.3]{CorOko1}
which involves non-weighted
modulation spaces as well as
\cite[Proposition 2.2]{Toft8B}
which involves weighted
modulation spaces. (See Definition
\ref{Def:OrliczSpaces2} for broader spectrum
of Orlicz spaces.)

\par

\begin{thm}\label{Thm:CordOkdjP3.3}
Let $f_1,f_2\in \maclS _{1/2}'(\rr d)$,
$\Phi _j$ and $\Psi _j$ be quasi-Young functions
$j=1,2$,
$\omega _0\in \mathscr P_E(\rr {4d})$
and $\omega _1,\omega
_2\in \mathscr P_E(\rr {2d})$.
Also let $\phi _1,\phi _2\in
\maclS _{1/2}(\rr d)$, and let $\phi
=V_{\phi _1}\phi _2$. Then the following is true:
\begin{enumerate}
\item if
\begin{equation}\label{Eq:omegarel1}
\omega _0(x,\xi ,\eta ,-y)\le C\omega
_1(y-x,\eta ) \omega _2(y,\xi +\eta ),
\qquad x,y,\xi ,\eta \in \rr d,
\end{equation}
for some constant $C>0$, then
\begin{equation}\label{Eq:STFTOrlModEst1}
\nm {V_\phi (V_{f_1}f_2)}
{L^{\Phi_1,\Phi _2,\Psi _1,\Psi _2}_{(\omega _0)}}
\le C\nm
{V_{\phi _1}f_1}
{L^{\Phi _1,\Psi _1}_{(\omega _1)}}
\nm {V_{\phi _2}f_2}
{L^{\Psi _2,\Phi _2}_{*,(\omega _2)}}
\text ;
\end{equation}

\vrum

\item if
\begin{equation}\label{Eq:omegarel2}
\omega _1(y-x,\eta ) \omega _2(y,\xi +\eta )
\le
C\omega _0(x,\xi ,\eta ,-y),
\qquad x,y,\xi ,\eta \in \rr d,
\end{equation}
for some constant $C$, then
\begin{equation}\label{Eq:STFTOrlModEst2}
\nm {V_{\phi _1}f_1}
{L^{\Phi _1,\Psi _1}_{(\omega _1)}}
\nm {V_{\phi _2}f_2}
{L^{\Psi _2,\Phi _2}_{*,(\omega _2)}}
\le
C\nm {V_\phi (V_{f_1}f_2)}
{L^{\Phi _1,\Phi _2,\Psi _1,\Psi _2}_{(\omega _0)}}
\text ;
\end{equation}

\vrum

\item if \eqref{Eq:omegarel1} and \eqref{Eq:omegarel2}
hold for some constant $C$, then
$f_1\in M^{\Phi _1,\Psi _1}_{(\omega _1)}(\rr d)$
and
$f_2\in W^{\Psi _2,\Phi _2}_{(\omega _2)}(\rr d)$,
if and only if
$V_{f_1}f_2
\in
M^{\Phi _1,\Phi _2,\Psi _1,\Psi _2}_{(\omega _0)}
(\rr {2d})$,
and
\begin{equation}\label{Eq:STFTOrlModEst3}
\nm {V_{f_1}f_2}
{M^{\Phi _1,\Phi _2,\Psi _1,\Psi _2}_{(\omega _0)}}
\asymp
\nm{f_1}{M^{\Phi _1,\Psi _1}_{(\omega _1)}}
\nm {f_2}{W^{\Psi _2,\Phi _2}_{(\omega _2)}}.
\end{equation}
\end{enumerate}
\end{thm}

\par

\begin{proof}
We shall mainly follow the proofs of Proposition 3.3
in \cite{CorOko1} and Proposition 2.2 in \cite{Toft8B}.

\par

It suffices to prove (1) and (2), and then we only prove
(1), since (2) follows by similar arguments.

\par

By Fourier's inversion formula we have
$$
|V_{\phi _1} f_1(-x-y,\eta )V_{\phi _2} f_2(-y,\xi +\eta
)|=|V_{\phi}(V_{f_1}f_2)(x,\xi ,\eta ,y)|
$$
(cf. e.{\,}g. \cite{Fol1,Gro1,Toft5,Toft15B}). Hence, if
$$
F_1(x,\xi ) = |V_{\phi _1} f_1(x,\xi )|\omega _1(x,\xi )
\quad \text{and}\quad
F_2(x,\xi ) = V_{\phi _2} f_2(x,\xi )\omega _2(x,\xi ),
$$
then
\begin{align*}
\nm
{V_\phi (V_{f_1}f_2)(\cdo ,\xi ,\eta ,y)
\omega _0(\cdo ,\xi ,\eta ,y)}
{L^{\Phi _1}(\rr {d})}
&\le C \nm {F_1(-y-\cdo ,\eta )F_2(-y,\xi+\eta )}
{L^{\Phi _1}(\rr d)}
\\[1ex]
&=
C\nm {F_1(\cdo ,\eta )}{L^{\Phi _1} (\rr d)}
F_2(-y,\xi+\eta ).
\end{align*}
By applying the $L^{\Phi _2}$ quasi-norm with respect to
the $\xi$-variables we obtain
\begin{multline*}
\nm
{V_\phi (V_{f_1}f_2)(\cdo ,\eta ,y)
\omega _0(\cdo ,\eta ,y)}
{L^{\Phi _1,\Phi _2}(\rr {2d})}
\\[1ex]
\le
C\nm {F_1(\cdo ,\eta )}{L^{\Phi _1}(\rr d)}
\nm {F_2(-y,\cdo +\eta )}{L^{\Phi _2}(\rr d)}
=
C\nm {F_1(\cdo ,\eta )}{L^{\Phi _1}(\rr d)}
\nm {F_2(-y,\cdo )}{L^{\Phi _2}(\rr d)}.
\end{multline*}
The result now follows by first applying
the $L^{\Psi _1}$ quasi-norm on the $\eta$-variables,
and then the $L^{\Psi _2}$ quasi-norm on the
$y$-variables.
\end{proof}

\par

\par

\begin{cor}\label{Cor:CordOkdjP3.3}
Let $f_1,f_2\in \maclS _{1/2}'(\rr d)$,
$\Phi$ and $\Psi$ be quasi-Young functions
and let
$\omega _0\in \mathscr P_E(\rr {4d})$
and $\omega _1,\omega
_2\in \mathscr P_E(\rr {2d})$
be such that
$$
\omega _0(x,\xi ,\eta ,-y)\asymp \omega
_1(y-x,\eta ) \omega _2(y,\xi +\eta ).
$$
Then
$f_1\in M^{\Phi ,\Psi}_{(\omega _1)}(\rr d)$
and
$f_2\in W^{\Psi ,\Phi}_{(\omega _2)}(\rr d)$,
if and only if
$V_{f_1}f_2
\in
M^{\Phi ,\Psi}_{(\omega _0)}
(\rr {2d})$,
and
$$
\nm {V_{f_1}f_2}
{M^{\Phi ,\Psi}_{(\omega _0)}}
\asymp
\nm{f_1}{M^{\Phi ,\Psi}_{(\omega _1)}}
\nm {f_2}{W^{\Psi ,\Phi}_{(\omega _2)}}.
$$
\end{cor}

\par

\subsection{Continuity
for pseudo-differential operators when
acting on Orlicz modulation spaces}

\par

Next we apply the previous results to deduce
continuity for pseudo-differential operators
with symbols in modulation spaces which act
on Orlicz modulation spaces. The involved
weight functions should satisfy
\begin{equation}\label{Eq:WeightCondPseudo}
\frac {\omega _2(x,\xi  )}{\omega _1
(y,\eta )} \lesssim \omega _0( x-A(x-y),A^*\xi
+(I-A^*)\eta ,\xi -\eta ,y-x ).
\end{equation}
The following result extend \cite[Theorem 5.1]{CorNic2}.

\par

\begin{thm}\label{Thm:PseudoCont1}
Let $A\in \GL (d,\mathbf R)$, $p,q\in [1,\infty ]$
be such that $q\le p$,
and let
$\Phi _j,\Psi _j : [0,\infty ]\to [0,\infty ]$, $j=1,2$,
be such that the following is true:
\begin{itemize}
\item if $p=1$, then $\Phi _j$ and $\Psi _j$ are Young
functions;

\vrum

\item if $p>1$, then $\Phi _j$ and $\Psi _j$ are
$p'$-steered Young functions 
which fulfill a local $\Delta _2$-condition,
and for some $r>0$, it holds
\begin{alignat}{3}
\Phi _1(t),\Phi _2(t) &\gtrsim t^{q'} & \quad
\Psi _1(t),\Psi _2(t) &\gtrsim t^{q'},
& \quad t &\in [0,r],
\label{Eq:WignerDistYoungFuncCond1Ver2}
\intertext{and}
\Phi _1^{-\&}(s)\Phi _2^{-\&}(s)
&\lesssim
s^{\frac 1{p'}+\frac 1{q'}}, &
\quad 
\Psi _1^{-\&}(s)\Psi _2^{-\&}(s)
&\lesssim
s^{\frac 1{p'}+\frac 1{q'}}, &
\quad s &\in [0,r].
\label{Eq:WignerDistYoungFuncCond2Ver2}
\end{alignat}
\end{itemize}
Also let $\omega _0\in \mascP _E(\rr {2d}
\oplus \rr {2d})$
and $\omega _1,\omega _2\in \mascP _E(\rr {2d})$
satisfy \eqref{Eq:WeightCondPseudo}.
If $a\in M^{p,q}_{(\omega _0)}(\rr {2d})$,
then $\op _A(a)$ from $\maclS _{1/2}(\rr d)$ to
$\maclS _{1/2}'(\rr d)$
extends uniquely to a continuous map from
$M^{\Phi _1,\Psi _1}_{(\omega_1)}(\rr d)$
to
$M^{\Phi _2^*,\Psi _2^*}_{(\omega _2)}(\rr d)$, and
\begin{equation}\label{Eq:(A.7)}
\nm {\op _A(a)}
{M^{\Phi _1,\Psi _1}_{(\omega _1)}
\to
M^{\Phi _2^*,\Psi _2^*} _{(\omega _2)}}
\lesssim
\nm a{M^{p,q} _{(\omega _0)}}.
\end{equation}

\par

Moreover, if in addition $a$ belongs to the
closure of $\maclS _{1/2}$ under
the $M^{p,q}_{(\omega _0 )}$ norm, then
$\op _A(a)\, :\, M^{\Phi _1,\Psi _1}_{(\omega _1)}(\rr d)
\to
M^{\Phi _2^*,\Psi _2^*}_{(\omega _2)}(\rr d)$ is compact.
\end{thm}

\par

\begin{proof}
First suppose that $p<\infty$. Then $q<\infty$, and it follows
that $\maclS _{1/2}(\rr {2d})$ is dense in
$M^{p,q}_{(\omega _0)}(\rr {2d})$.
Let $f,g\in \maclS _{1/2}'(\rr d)$ and
$a\in \maclS _{1/2}(\rr {2d})$. Then
\eqref{Eq:PseudoWignerLink}
and Theorem \ref{Thm:WignerDistOrlicModSpace}
gives
\begin{equation}\label{Eq:PseudoActionEst}
\begin{aligned}
|(\op _A(a)f,g)|
&\asymp
|(a,W^A_{g,f})|
\lesssim
\nm a{M^{p,q}_{(\omega _0)}}
\nm {W^A_{g,f}}{M^{p',q'}_{(1/\omega _0)}}
\\[1ex]
&\lesssim
\nm a{M^{p,q}_{(\omega _0)}}
\nm f{M^{\Phi _1,\Psi _1}_{(\omega _1)}}
\nm g{M^{\Phi _2,\Psi _2}_{(1/\omega _2)}},
\end{aligned}
\end{equation}
and by duality it follows that $\op _A(a)$ from
$\maclS _{1/2}'(\rr d)$ to $\maclS _{1/2}(\rr d)$ restricts
to a continuous map from 
$M^{\Phi _1,\Psi _1}_{(\omega _1)}(\rr d)$
to
$M^{\Phi _2^*,\Psi _2^*}_{(\omega _2)}(\rr d)$,
and that \eqref{Eq:PseudoActionEst} gives
\eqref{Eq:(A.7)}. The result now follows in
this case by \eqref{Eq:PseudoActionEst} and the
fact that $\maclS _{1/2}(\rr {2d})$ is dense in
$M^{p,q}_{(\omega _0)}(\rr {2d})$.

\par

Next suppose that $p=\infty$ and $q<\infty$. If
$a\in \maclS _{1/2}(\rr {2d})$, then
\eqref{Eq:PseudoActionEst} implies that \eqref{Eq:(A.7)}
holds in this case as well. By Hahn-Banach's theorem it
follows that the definition of $\op _A(a)$ is extendable
to any $a\in M^{p,q}_{(\omega _0)}(\rr {2d})$, and that
\eqref{Eq:(A.7)} still holds. The uniqueness of the extension
now follows from the fact that 
$\maclS _{1/2}(\rr {2d})$
is dense in $M^{p,q}_{(\omega _0)}(\rr {2d})$ with
respect to the narrow convergence, when $q<\infty$
(see \cite{Toft28}). 

\par

Finally, if $p=q=\infty$, then
\eqref{Eq:WignerDistYoungFuncCond2} implies that
$$
\Phi _j(t)\asymp t
\quad \text{and}\quad
\Psi _j(t)\asymp t,
\quad
j=1,2,
$$
giving that $M^{\Phi _1,\Psi _1}_{(\omega _1)}(\rr d)
=M^{1,1}_{(\omega _1)}(\rr d)$ and
$M^{\Phi _2^*,\Psi _2^*}_{(\omega _2)}(\rr d)=
M^{\infty,\infty}_{(\omega _2)}(\rr d)$. The result now
follows by choosing $p=q=\infty$
in \cite[Theorem 2.2]{Toft15B}.
\end{proof}

\par

As a special case we obtain the following extension of
Proposition \ref{Prop:PseudoContIntro1} in the
introduction. The details are left for the reader.

\par

\renewcommand{\rubrik}{Proposition \ref{Prop:PseudoContIntro1}$'${\!}}

\par

\begin{tom}
Let  $A\in \GL (d,\mathbf R)$,
$p,q\in [1,\infty ]$ be such that $q\le p$ and $p>1$,
$\omega _0\in \mascP _E(\rr {2d}
\oplus \rr {2d})$
and $\omega _1,\omega _2\in \mascP _E(\rr {2d})$
satisfy \eqref{Eq:WeightCondPseudo}.
Also let
$\Phi _j,\Psi _j : [0,\infty ]\to [0,\infty ]$, $j=1,2$,
be such that  $t\mapsto \Phi _j(t^{\frac 1{p'}})$
and
$t\mapsto \Psi _j(t^{\frac 1{p'}})$ are
Young functions 
which fulfill the $\Delta _2$-condition,
and
\begin{alignat*}{3}
\Phi _1(t),\Phi _2(t) &\gtrsim t^{q'} & \quad
\Psi _1(t),\Psi _2(t) &\gtrsim t^{q'},
& \quad t &\ge 0,
\intertext{and}
\Phi _1^{-1}(s)\Phi _2^{-1}(s)
&\lesssim
s^{\frac 1{p'}+\frac 1{q'}}, &
\quad
\Psi _1^{-1}(s)\Psi _2^{-1}(s)
&\lesssim
s^{\frac 1{p'}+\frac 1{q'}}, &
\quad s &\ge 0.
\end{alignat*}
If $a\in M^{p,q}_{(\omega _0)}(\rr {2d})$,
then $\op _A(a)$ is continuous from
$M^{\Phi _1,\Psi _1}_{(\omega _1)}(\rr d)$
to
$M^{\Phi _2^*,\Psi _2^*}_{(\omega _2)}(\rr d)$.
\end{tom}

\par

By similar type of duality arguments,
using Theorem \ref{Thm:CordOkdjP3.3} instead of
Theorem \ref{Thm:WignerDistOrlicModSpace},
we obtain the following extension of
\cite[Theorem 2.1]{Toft15B}. Here we observe that
\eqref{Eq:WeightCondPseudo} takes the form
\begin{equation}\tag*{(\ref{Eq:WeightCondPseudo})$'$}
\frac {\omega _2(x,\xi )}{\omega _1(y,\eta )} 
\lesssim
\omega (x,\xi +\eta ,\xi -\eta ,y-x)
\end{equation}
when $A=0$.

\par

\begin{thm}\label{Thm:Wpseudos}
Let $\Phi$ and $\Psi$ be Young functions
which satisfy local $\Delta _2$-condition,
and let $\omega \in \mascP _E(\rr {4d})$
and $\omega _1,\omega _2\in \mascP _E (\rr {2d})$
be such that \eqref{Eq:WeightCondPseudo}$'$
holds. Also let
$a\in W^{\Psi ,\Phi}_{(\omega )}(\rr {2d})$.
Then the definition of $\op _0(a)$
from $\maclS _{1/2}(\rr d)$ to
$\maclS '(\rr d)$ extends uniquely to a
continuous map from
$M^{\Phi ^*,\Psi ^*}_{(\omega _1)}(\rr d)$
to $W^{\Psi ,\Phi}_{(\omega _2)}(\rr
d)$, and
\begin{equation}
\nm {\op _0(a)f}{W^{\Psi ,\Phi}_{(\omega _2)}}
\lesssim
\nm {a}{W^{\Psi ,\Phi}_{(\omega )}}
\nm f{M^{\Phi ^*,\Psi ^*}_{(\omega _1)}},
\qquad
a\in W^{\Psi ,\Phi}_{(\omega )}(\rr {2d}),
\ 
f\in M^{\Phi ^*,\Psi ^*}_{(\omega _1)}(\rr d).
\end{equation}
\end{thm}

\par

\begin{proof}
We shall follow the proof of Theorem 2.1 in
\cite{Toft15B}.
We may assume that equality holds in
\eqref{Eq:WeightCondPseudo}$'$.
We start to prove the result in the case $\Phi ^*(t)>0$
and $\Psi ^*(t)>0$ when $t>0$. Then we may replace
$\Phi$ and $\Psi$ such that the Orlicz modulation
spaces are the same and $\Phi$ and $\Psi$
satisfy (global) $\Delta _2$-conditions.

\par

Let
$$
\omega _0(x,\xi ,\eta ,y ) = \omega (-y,\eta ,\xi ,-x)^{-1},
$$
$a\in W^{\Psi ,\Phi}_{(\omega )}(\rr {2d})$ and $f,g\in
\maclS _{1/2}(\rr d)$. Then $\op _0(a)f$
makes sense as an element in
$\maclS _{1/2}'(\rr d)$. 

\par

By Theorem \ref{Thm:CordOkdjP3.3}  we get
\begin{equation}\label{Vfgigen}
\nm {V_fg}{M^{\Phi ^*,\Psi ^*}_{(\omega _0)}}
\lesssim
\nm f{M^{\Phi ^*,\Psi ^*}_{(\omega
_1)}}\nm g{W^{\Psi ^*,\Phi ^*}_{(1/\omega _2)}}.
\end{equation}

\par

Furthermore, if $T$ is the torsion operator
defined by $TF(x,\xi )=F(-\xi ,x)$ when
$F\in \maclS _{1/2}'(\rr {2d})$, then it
follows by Fourier's inversion formula that
$$
(V_{\phi }(T\widehat a))(x,\xi ,\eta ,y )
=
e^{-i(\scal x\eta +\scal y\xi )}
(V_{T\widehat \phi }a)(-y,\eta ,\xi
,-x).
$$
This gives
\begin{equation*}
|(V_{\phi}(T\widehat a))(x,\xi ,\eta ,y )\omega _0(x,\xi ,\eta
,y)^{-1}|
\\[1ex]
=|(V_{\phi _1}a)(-y,\eta ,\xi ,-x)\omega (-y,\eta ,\xi ,-x)|,
\end{equation*}
when $\phi _1=T\widehat \phi$. Hence, by
applying the $L^{\Phi ,\Psi}$ norm we obtain
$$
\nm {T\widehat a}{M^{\Phi ,\Psi }_{(1/\omega _0)}}
=\nm a{W^{\Psi ,\Phi}_{(\omega )}}.
$$

\par

It  now follows from \eqref{Vfgigen} that
\begin{equation}\label{estimates1}
\begin{aligned}
|(\op _0(a)f,g)|
&=
(2\pi )^{-d/2}|(T\widehat a,V_{\overline g}\overline f)|
\\[1ex]
&\lesssim
\nm {T\widehat a}{M^{\Phi ,\Psi}_{(1/\omega _0)}}
\nm {V_{f}g}{M^{\Phi ^*,\Psi ^*}_{(\omega _0)}}
\lesssim \nm a{W^{\Psi ,\Phi}_{(\omega )}}
\nm f{M^{\Phi ^*,\Psi ^*}_{(\omega _1)}}\nm
g{W^{\Psi ^*,\Phi ^*}_{(1/\omega _2)}}.
\end{aligned}
\end{equation}
The result now follows by the facts that
$\maclS _{1/2}(\rr d)$ is dense
in $M^{\Phi ^*,\Psi ^*}_{(\omega _1)}(\rr d)$, and
that the dual of
$W^{\Psi ^*,\Phi ^*}_{(1/\omega _2)}$ is
$W^{\Phi ,\Psi}_{(\omega _2)}$ when
$\Phi ^*$ and $\Psi ^*$ satisfies the
$\Delta _2$-condition.

\par

If instead $\Phi ^*(t)=0$ and $\Psi (t)>0$,
or $\Phi (t)>0$ and $\Psi ^*(t)=0$, when $t>0$
is close to origin,
then let $f\in M^{\Phi ^*,\Psi ^*}_{(\omega _1)}$
and $a\in \maclS _{1/2}(\rr {2d})$.
Then $\op _0(a)f$ makes sense as an element in
$\maclS _{1/2}(\rr d)$, and from the first part
of the proof it follows that
\eqref{estimates1} still holds. The result now follows
by duality and
the fact that $\mathscr S(\rr {2d})$ is dense in
$W^{\Psi ,\Phi}_{(\omega
)}(\rr {2d})$, since it follows from the assumptions
that $\Phi$ and $\Psi$ fulfill the
$\Delta _2$-condition.

\par

It remains to consider the case when
$\Phi (t)=\Psi ^*(t)=0$ and the case when
$\Phi ^*(t)=\Psi (t)=0$ when $t> 0$ is near origin.
In this case, we have
$$
W^{\Psi ,\Phi} = W^{1,\infty}
\quad \text{and}\quad
M^{\Phi ^*,\Psi ^*} = M^{1,\infty},
$$
or
$$
W^{\Psi ,\Phi} = W^{\infty ,1}
\quad \text{and}\quad
M^{\Phi ^*,\Psi ^*} = M^{\infty ,1}.
$$
The result then follows by letting
$p=q'=\infty$ or $p=q'=1$ in
\cite[Theorem 3.9]{TeoTof} or in the proof of
\cite[Theorem 2.1]{Toft15B}.
\end{proof}

\par

\begin{example}\label{Example:ContPseudo}
Let $A\in \GL (d,\mathbf R)$, $p>2$, $a\in M^{p,p'}(\rr {2d})$
and $\Phi$ be a Young
function which fullfils \eqref{Eq:SpecYoungFunc}. That is,
we let $q=p'$ in our results. Then
it follows that $\Phi$ fullfils a local $\Delta _2$-condition,
$$
\Phi (t) \gtrsim t^{p}=t^{q'}
\quad \text{and}\quad
\Phi ^{-1}(s)^2
\lesssim
s=s^{\frac 1{p'}+\frac 1{q'}}.
$$
Hence the hypothesis in Propositions \ref{Prop:PseudoContIntro1}
and \ref{Prop:PseudoContIntro1}$'$ (as well as in
Theorem \ref{Thm:PseudoCont1}) are fulfilled with
$$
\Phi _1=\Phi _2=\Psi _1=\Psi _2=\Phi . 
$$
It now follows from any of these results that
\begin{equation}\label{Eq:ExContPseudo}
\op _A(a) : M^\Phi (\rr d)\to M^\Phi (\rr d)
\end{equation}
is continuous.

\par

We also observe that if instead $a$ belongs to $M^2(\rr {2d})$,
which is near $M^{p,p'}(\rr {2d})$ when $p>2$ is closed to $2$,
then the map \eqref{Eq:ExContPseudo} may be discontinuous
(cf. Remark \ref{Rem:DiscontPseudo} in the end of the next section).
\end{example}

\par

\section{An Orlicz modulation space suitable for
the entropy functional}\label{sec3}

\par

In this section we show that the entropy
functional in \eqref{Eq:SpecOrlModSpIntro}
is continuous on the
modulation spaces $M^p(\rr d)$, $1\le p<2$,
and the Orlicz modulation
space $M^\Phi (\rr d)$ with $\Phi (t)=-t^2\log t$ near origin.
For completeness we also give a proof of that the same
functional is discontinuous on $M^2(\rr d)=L^2(\rr d)$.
(Cf. Theorem \ref{Thm:EntropyCont}.)
In order to reach such properties we need to prove some
preparing results which might be of independent interest.
For example we deduce estimates for entropy functionals
when changing window functions (see Lemma
\ref{Lem:DiscontEntrOnL2}).

\par

We observe that the entropy functional
\eqref{Eq:SpecOrlModSpIntro} can be written as
\begin{equation}\tag*{(\ref{Eq:EntropyFunc})$'$}
\begin{aligned}
E(f) &= E_\phi (f)
\\[1ex]
&\equiv
-\iint _{\rr {2d}}|V_\phi f(x,\xi )|^2
\log |V_\phi f(x,\xi )|^2\, dxd\xi
+
\nm {\phi}{L^2}^2\nm {f}{L^2}^2\log (\nm {\phi}{L^2}^2\nm {f}{L^2}^2),
\end{aligned}
\end{equation}
by using Moyal's identity
\begin{equation}\label{Eq:Moyal}
\nm {V_\phi f}{L^2} =\nm f{L^2}\nm \phi{L^2}
\end{equation}
(see e.{\,}g. \cite{Gro1}). In particular, if $\nm f{L^2}=\nm \phi{L^2}=1$
which is a common condition in the applications, the entropy
of $f$ becomes
\begin{equation}\tag*{(\ref{Eq:EntropyFunc})$''$}
E(f) = E_\phi (f)
=
-\iint _{\rr {2d}}|V_\phi f(x,\xi )|^2
\log |V_\phi f(x,\xi )|^2\, dxd\xi ,
\quad
\nm f{L^2}=\nm \phi{L^2}=1
\end{equation}
(see e.{\,}g. \cite{Lie1,LieSol}). For general $f,\phi \in L^2(\rr d)$
we observe that the entropy possess homogeneity properties
of the form
\begin{equation}\label{Eq:EntrHom}
E_{\lambda \phi}(f) = E_{\phi}(\lambda f) = |\lambda |^2E_\phi (f),
\qquad
f,\phi \in L^2(\rr d),\ \lambda \in \mathbf C .
\end{equation}
In fact, Moyal's identity gives
\begin{align*}
E_{\lambda \phi}(f) &= E_{\phi}(\lambda f)
\\[1ex]
&=
-\iint _{\rr {2d}}|\lambda |^2 |V_\phi f(x,\xi )|^2
\log (|\lambda ||V_\phi f(x,\xi )|)^2\, dxd\xi
+
|\lambda |^2\nm {\phi}{L^2}^2\nm {f}{L^2}^2
\log (|\lambda |^2\nm {\phi}{L^2}^2\nm {f}{L^2}^2)
\\[1ex]
&=
|\lambda |^2
\left (-\iint _{\rr {2d}} |V_\phi f(x,\xi )|^2
\log |V_\phi f(x,\xi )|^2\, dxd\xi
+
\nm {\phi}{L^2}^2\nm {f}{L^2}^2
\log (\nm {\phi}{L^2}^2\nm {f}{L^2}^2)
\right )
\\
&+
(\log |\lambda |^2)(\nm \phi{L^2}^2\nm f{L^2}^2-\nm {V_\phi f}{L^2} ^2)
\\[1ex]
&=
|\lambda |^2E_\phi (f).
\end{align*}

\par

In order to discuss continuity for the entropy functional, we
restrict ourself and assume that the window functions belong
to the subspace $M^1(\rr d)$ of $L^2(\rr d)$.
The main result of the section is the following.

\par

\begin{thm}\label{Thm:EntropyCont}
Let $\Phi$ be a Young function which satisfies
\eqref{Eq:SpecOrlModSpIntro},
$\phi \in M^1(\rr d)\setminus 0$ and let
$E_\phi$ be as in \eqref{Eq:EntropyFunc}. Then the following
is true:
\begin{enumerate}
\item $E_\phi$ is continuous on $M^p(\rr d)$ and on $M^\Phi (\rr d)$,
$1\le p<2$;

\vrum

\item $E_\phi$ is discontinuous on $M^2(\rr d)$.
\end{enumerate}
\end{thm}

\par

We need some preparations for the proof of
Theorem \ref{Thm:EntropyCont}. First we observe that
$M^{\Phi}(\rr d)$ is in some sense close to $M^2(\rr d)$.

\par

\begin{lemma}\label{Lem:OrlModCloseM2}
Let $\Phi$ be a Young function which satisfies
\eqref{Eq:SpecOrlModSpIntro}. Then
\begin{equation}\label{Eq:NarrowSpaces1}
M^p(\rr d)\subseteq M^\Phi (\rr d) \underset{\text{dense}}
\subseteq M^2(\rr d),
\qquad p<2,
\end{equation}
with continuous inclusions, and
\begin{equation}\label{Eq:NarrowSpaces2}
\lim _{p\to 2-}\nm f{M^p}=\nm f{M^2},
\quad \text{when}\quad
f\in M^{p_0}(\rr d),
\ \text{for some}\ 1\le p_0<2.
\end{equation}
\end{lemma}

\par

For the limit in \eqref{Eq:NarrowSpaces2} it is understood that the
same window function is used in the modulation space norms.

\par

\begin{proof}
By Proposition \ref{Prop:OrliczModInvariance} it follows
that $M^\Phi (\rr d)$
is independent of the choice of $\Phi$ outside the
interval $[0,e^{-\frac 23}]$. It is therefore no restriction
to assume that $\Phi$ is given by
\begin{equation}\tag*{(\ref{Eq:EntropyFunc})$'$}
\Phi (t)
=
\begin{cases}
-t^2\log t , & t\in [0,e^{-\frac 23}],
\\[1ex]
\frac 13e^{-\frac 23}(t+e^{-\frac 23}),
& t\in (e^{-\frac 23},\infty ),
\\[1ex]
\infty , & t=\infty ,
\end{cases}
\end{equation}
which is obviously a Young function.

\par

By Remark \ref{Rem:PhiLeb} and the limits
$$
\lim _{t\to 0+}\frac {t^2}{\Phi (t)}
=
\lim _{t\to 0+}-\frac {t^2}{t^2\log t}=0
\quad \text{and}\quad
\lim _{t\to 0+}\frac {t^2}{\Phi (t)}
=
\lim _{t\to 0+}-\frac {t^p}{t^2\log t}=\infty ,
$$
when $p<2$, it follows from
Proposition \ref{Prop:OrliczModInvariance}
that the inclusions in \eqref{Eq:NarrowSpaces1} holds
and are continuous.
Since $M^p(\rr d)$ is dense in $M^2(\rr d)$, it follows that
also $M^\Phi (\rr d)$ is dense in $M^2(\rr d)$.

\par

The limit in \eqref{Eq:NarrowSpaces2} follows by straight-forward
computations in measure theory (cf. e.{\,}g. the exercise part of
Chapter 3 in \cite{Rud}).
\end{proof}

\par

\begin{rem}\label{Rem:NonEqualOrlModSp}
Let $\Phi$ be a Young function which satisfies
\eqref{Eq:SpecOrlModSpIntro}.
A consequence of Theorem \ref{Thm:EntropyCont},
Lemma \ref{Lem:OrlModCloseM2} and the open
mapping theorem is that $M^\Phi (\rr d)\subsetneq M^2(\rr d)$.
\end{rem}

\par

Next we show
that $E_\phi$ is well-defined and finite on $M^\Phi (\rr d)$.

\par

\begin{lemma}\label{Lem:EntropyCont}
Let $\Phi$ be a Young function which satisfies
\eqref{Eq:SpecOrlModSpIntro}, $f,\phi \in M^2(\rr d)$.
Then the following is true:
\begin{enumerate}
\item $|V_\phi f(x,\xi )|^2 \log _+|V_\phi f(x,\xi )|\in L^1(\rr {2d})$
and
$$
\iint _{\rr {2d}}
|V_\phi f(x,\xi )|^2 \log |V_\phi f(x,\xi )| \, dxd\xi \in [-\infty ,\infty )
\text ;
$$

\vrum

\item if in addition $f\in M^\Phi (\rr d)$ and $\phi \in M^1(\rr d)$,
then
$$
\iint _{\rr {2d}}
|V_\phi f(x,\xi )|^2 \big | \log |V_\phi f(x,\xi )| \big |\, dxd\xi <\infty .
$$
\end{enumerate}
\end{lemma}

\par

\begin{proof}
The assertion (1) follows from the fact that $V_\phi f \in
L^2(\rr {2d})\cap L^\infty (\rr {2d})$, in view of Moyal's
identity and the embedding $M^2(\rr d)\subseteq M^\infty (\rr d)$,
ensured by Proposition \ref{Prop:BasicPropOrlModSp1} (3).

\par

Since (2) is obviously true when $f$ or $\phi$ are identically equal to zero,
we may assume that $f\in M^\Phi (\rr d)\setminus 0$ and
$\phi \in M^1(\rr d)\setminus 0$.
Let $\phi$ be chosen as the window function in our modulation space norms,
$C>1$ be a fixed constant and for every
$f\in M^\Phi (\rr d)$, choose the number $\lambda =\lambda _f$
such that
$$
\nm f{M^\Phi} < \lambda <C\nm f{M^\Phi}.
$$
For conveniency we also let $F=V_\phi f$,
\begin{align*}
\Omega _1 &= \sets {(x,\xi )\in \rr {2d}}{|F(x,\xi )|\le \lambda e^{-\frac 23}},
\\[1ex]
\Omega _2 &= \sets {(x,\xi )\in \rr {2d}}
{\lambda e^{-\frac 23}\le |F(x,\xi )|\le \lambda}
\intertext{and}
\Omega _3 &= \sets {(x,\xi )\in \rr {2d}}{|F(x,\xi )|\ge \lambda}.
\end{align*}
Then
$$
\left |
-\iint _{\rr {2d}}|F(x,\xi )|^2\log |F(x,\xi )|\, dxd\xi
\right |
\le
\sum _{k=1}^4J_k(f),
$$
where
\begin{align*}
J_k(f)
&=
\lambda ^2
\left |
\iint _{\Omega _k}
\left (
\frac {|F(x,\xi )|}\lambda
\right )^2
\log
\left  (
\frac {|F(x,\xi )|}\lambda
\right )
\, dxd\xi 
\right |,
\quad k=1,2,3,
\intertext{and}
J_4(f)
&=
|\log \lambda |\cdot \nm f{M^2}^2,
\end{align*}
and the result follows if we prove
\begin{equation}\label{Eq:EntrIntFinitePart}
J_k(f) <\infty ,
\qquad k=1,2,3,4.
\end{equation}

\par

By the definition of $\Phi$ and the Orlicz modulation space
norm, we have
$$
J_1(f)\le \lambda ^2 \le C^2\nm f{M^\Phi}^2<\infty ,
$$
which shows that \eqref{Eq:EntrIntFinitePart} holds
for $k=1$.

\par

In order to prove \eqref{Eq:EntrIntFinitePart} for $k=2$
and $k=3$ we recall that
$\nm f{M^\infty}\lesssim \nm f{M^2}\lesssim \nm f{M^\Phi}$,
which implies that $|F(x,\xi )|\lesssim \nm f{M^\Phi}$. On
the other hand,  
$|F(x,\xi )|\gtrsim \nm f{M^\Phi}$ when
$(x,\xi )\in \complement \Omega _1$.
A combination of these relations yields
$|F(x,\xi )|\asymp \nm f{M^\Phi}$ when $(x,\xi )\in \complement
\Omega _1$,
which implies that the logarithm in the integral expression
of $J_k(f)$
is bounded when $k=2$ or $k=3$. This gives
$$
0\le J_k(f)
\lesssim
\iint _{\complement \Omega _1}
|F(x,\xi )|^2\, dxd\xi 
\le
\nm f{M^2}^2\lesssim \nm f{M^\Phi}^2,
\quad
k=2,3,
$$
and \eqref{Eq:EntrIntFinitePart}
follows in the cases $k=2$ and $k=3$.

\par

Finally, for $J_4(f)$ we have
$$
0\le J_3(f)
\le
|\log \lambda |\nm f{M^2}^2
\lesssim
|\log \lambda | \nm f{M^\Phi}^2<\infty ,
$$
and the result follows.
\end{proof}

\par

The next lemma gives an essential step
when deducing the asserted continuity in
Theorem \ref{Thm:EntropyCont}

\par

\begin{lemma}\label{Lem:PartEntropyCont}
Let $\Phi$ and $\phi$ be the same as in
Lemma \ref{Lem:EntropyCont}. Then
\begin{equation}\label{Eq:EntrBasicCont}
M^\Phi (\rr d)\ni f\mapsto \iint _{\rr {2d}}
|V_\phi f(x,\xi )|^2 \big | \log |V_\phi f(x,\xi )| \big |\, dxd\xi
\end{equation}
is continuous near origin.
\end{lemma}

\par

\begin{proof}
The result follows if we prove
\begin{equation}\label{Eq:EntrIntLimit}
-\iint _{\rr {2d}}|V_\phi f(x,\xi )|^2
\left |
\log |V_\phi f(x,\xi )|
\right |
\, dxd\xi \to 0
\quad \text{as}\quad
\nm f{M^\Phi}\to 0,\ f\in M^\Phi (\rr d).
\end{equation}

\par

Let $\phi$, $C$, $\lambda$ and $J_k(f)$ be the same as in
the proof of Lemma \ref{Lem:EntropyCont}.
Then
$$
\left |
\iint _{\rr {2d}}|V_\phi f(x,\xi )|^2
\left |
\log |V_\phi f(x,\xi )|
\right |
\, dxd\xi
\right |
\le
\sum _{k=1}^4J_k(f),
$$
and \eqref{Eq:EntrIntLimit} follows if we prove
\begin{equation}\label{Eq:EntrIntLimitPart}
J_k(f) \to 0
\quad \text{as}\quad
\nm f{M^\Phi}\to 0,\ f\in M^\Phi (\rr d),
\qquad k=1,2,3,4.
\end{equation}

\par

By the definition of $\Phi$ and the Orlicz modulation space
norm, we have
$$
0\le J_1(f)\le \lambda ^2 \le C^2\nm f{M^\Phi}^2\to 0
\quad \text{as}\quad \nm f{M^\Phi}\to 0,
$$
which shows that \eqref{Eq:EntrIntLimitPart} holds
for $k=1$.

\par

In order to prove \eqref{Eq:EntrIntLimitPart} for $k=2$
and $k=3$ we recall from the proof of
Lemma \ref{Lem:EntropyCont} that the logarithm in
\eqref{Eq:EntrIntLimitPart}
is bounded when $k=2$ or $k=3$. This gives
$$
0\le J_k(f)
\lesssim
\iint _{\Omega _k}
|V_\phi f(x,\xi )|^2\, dxd\xi 
\le
\nm f{M^2}^2\lesssim \nm f{M^\Phi}^2\to 0
$$
as $\nm f{M^\Phi}\to 0$, and \eqref{Eq:EntrIntLimitPart}
follows in the cases $k=2$ and $k=3$.

\par

For $J_4(f)$ with $\nm f{M^\Phi}\le 1$ we have
$$
0\le J_4(f)
\le
\nm f{M^2}^2|\log \nm f{M^\Phi}|
\lesssim
\nm f{M^\Phi}^2|\log \nm f{M^\Phi}| \to 0
$$
as $\nm f{M^\Phi}\to 0$. This gives
\eqref{Eq:EntrIntLimitPart}
in the case $k=4$, and
\eqref{Eq:EntrIntLimitPart} follows,
and we have proved that the map \eqref{Eq:EntrBasicCont}
is continuous at origin.
\end{proof}

%
%
%

\par

The next lemma concerns estimates
for $E_\phi$ in transitions between different
window functions $\phi$. The result
is needed in the proof of discontinuity
of $E_\phi$ on $M^2(\rr d)$. 

\par

\begin{lemma}\label{Lem:DiscontEntrOnL2}
Let $\Phi$ be a Young function and
$\phi ,\psi \in M^1(\rr d)\setminus 0$.
Then there is a constant $C$ which only depends on
$\phi$ and $\psi$ such that
\begin{equation}\label{Eq:DiscontEntrOnL2}
E_\phi (f)\le C(E_\psi (f) +\nm f{L^2}^2),
\qquad
f\in M^2(\rr d).
\end{equation}
\end{lemma}

\par

\begin{proof}
Let
$$
F_1 = |V_\phi f| ,
\quad
F_2 = |V_\psi f|
\quad \text{and}\quad
H=|V_\phi \psi |.
$$
We recall that $\nm H{L^1}\asymp \nm \phi {M^1}\nm \psi {M^1}<\infty$
in view of \cite[Proposition 12.1.2]{Gro1}. Since
\begin{equation}\label{Eq:HomogenMaps}
\mascS (\rr d)\ni \phi \mapsto E_\phi (f_0),
\quad
M^2(\rr d)\ni f \mapsto E_{\phi _0} (f)
\quad \text{and}\quad
L^2(\rr d)\ni f\mapsto \nm f{L^2}^2
\end{equation}
are positively homogeneous of order $2$
for every fixed $\phi _0\in \mascS (\rr d)\setminus 0$
and $f_0\in M^2(\rr d)\setminus 0$,
we reduce ourselves to the case when $\nm H{L^1}=1$.

\par

We recall that for some constant $C_1>0$ we have
$\nm {V_\psi f}{L^\infty}\le C_1\nm {V_\psi f}{L^2}$ for
every $f\in M^2(\rr d)$.
First assume that $\phi$, $\psi$ and $f$ are chosen
such that $\nm H{L^1}=1$ and
\begin{equation}\label{Eq:ExactNormWithF2}
\nm {F_2}{L^2}=\nm {V_\psi f}{L^2}=e^{-\frac 23}/C_1.
\end{equation}
Then $0\le F_2(x,\xi )\le \nm {V_\psi f}{L^\infty}\le e^{-\frac 23}$.
By \cite[Lemma 11.3.3]{Gro1} we obtain
\begin{equation}\label{Eq:F1F2Ests}
\begin{aligned}
0\le F_1(x,\xi )\le (F_2*H)(x,\xi )
&=
\iint _{\rr {2d}}F_2(x-y,\xi -\eta )\, d\mu (y,\eta)
\\[1ex]
&\le
\nm {F_2}{L^\infty}\nm H{L^1}
\le
e^{-\frac 23}.
\end{aligned}
\end{equation}
Here $\mu$ is the positive measure given by
$d\mu (y,\eta )=H(y,\eta )\, dyd\eta$, giving that
$$
\int _{\rr {2d}}d\mu =\nm H{L^1}=1.
$$

\par

Since $t\mapsto \fy (t)=-t^2\log t$ is increasing and convex
on $[0,e^{-\frac 23}]$, it follows from
\eqref{Eq:F1F2Ests} and Jensen's inequality that
\begin{align}
E_{0,\phi}(f)
&\equiv
-\iint _{\rr {2d}} F_1(x,\xi )^2\log F_1(x,\xi )\, dxd\xi
=
\iint _{\rr {2d}} \fy (F_1(x,\xi ))\, dxd\xi 
\notag
\\[1ex]
&\le
\iint _{\rr {2d}} \fy \left (
\iint _{\rr {2d}}F_2(x-y,\xi -\eta ) \, d\mu (y,\eta )\right )\, dxd\xi
\notag
\\[1ex]
&\le
\iint _{\rr {2d}} \left (
\iint _{\rr {2d}}\fy \left ( F_2(x-y,\xi -\eta )\right ) \, d\mu (y,\eta )
\right ) \, dxd\xi
\notag
\\[1ex]
&=
\nm H{L^1}
\iint _{\rr {2d}}\fy \left ( F_2(x,\xi)\right ) \, dxd\xi
=
E_{0,\psi}(f)
\label{Eq:SpecEntrFuncEst}
\end{align}

\par

Now choose $C_0\ge \max (e,e^{\frac 53}C_1)$ such that
$$
\nm {V_\phi f}{L^2}\le C_0\nm {V_\psi f}{L^2}
\quad \text{and}\quad
\nm {V_\psi f}{L^2}\le C_0\nm f{L^2},
$$
for every $f\in M^2(\rr d)$, which is possible because
$$
f\mapsto \nm {V_\phi f}{L^2}
\quad \text{and}\quad
f\mapsto \nm {V_\psi f}{L^2}
$$
are two equivalent norms for $M^2(\rr d)=L^2(\rr d)$.
Then $\log C_0\ge 1$.
A combination of \eqref{Eq:ExactNormWithF2}
and \eqref{Eq:SpecEntrFuncEst}
gives $\log (C_0\nm {F_2}{L^2})\ge 1$ and
\begin{align*}
E_\phi (f)
&= 
2E_{0,\phi}(f)+2\nm {F_1}{L^2}^2 \log \nm {F_1}{L^2}
\\[1ex]
&\le
2E_{0,\psi}(f)+2(C_0\nm {F_2}{L^2})^2 \log (C_0\nm {F_2}{L^2})
\\[1ex]
&=
2E_{0,\psi}(f)+2C_0^2\nm {F_2}{L^2}^2 \log \nm {F_2}{L^2}
+2(C_0^2\log C_0)\nm {F_2}{L^2}^2
\\[1ex]
&\le
C(E_{0,\psi}(f)+\nm {F_2}{L^2}^2) \log \nm {F_2}{L^2}
+\nm f{L^2}^2),
\end{align*}
when $C=2C_0^4\log C_0$. Hence
\eqref{Eq:DiscontEntrOnL2} follows under the additional
condition \eqref{Eq:ExactNormWithF2}. The estimate
\eqref{Eq:DiscontEntrOnL2} now follows for general 
$f\in M^2(\rr d)$ by the homogeneity of the mappings
in \eqref{Eq:HomogenMaps}, and the result follows.
\end{proof}

\par

\begin{proof}[Proof of Theorem \ref{Thm:EntropyCont}]
We choose $\Phi$ as in \eqref{Eq:EntropyFunc}$'$.
First we prove the continuity for
$E_\phi$ on $M^\Phi (\rr d)$ at origin.

\par

By \eqref{Eq:NarrowSpaces1} it follows that
$\nm f{M^2}\le C\nm f{M^\Phi}$,
for some constant $C\ge 1$ which is independent of
$f\in M^\Phi (\rr d)$. Hence, for $f\in M^\Phi (\rr d)$ with
$\nm f{M^\Phi}$ being small enough we obtain
\begin{align*}
\left |
\nm f{M^2}^2\log \nm f{M^2}
\right |
&\le
C^2\left |
\nm f{M^\Phi}^2\log (C\nm f{M^\Phi})
\right |
\\[1ex]
&\le
C^2\left (\left |
\nm f{M^\Phi}^2\log (\nm f{M^\Phi})
\right | +(\log C)\nm f{M^\Phi}^2\right )
\to 0
\end{align*}
as $\nm f{M^\Phi}\to 0$. A combination of the latter
continuity and \eqref{Eq:EntrIntLimit} now gives
\begin{equation*}
\begin{gathered}
E_\phi (f)
=
-\iint _{\rr {2d}}|V_\phi f(x,\xi )|^2\log |V_\phi f(x,\xi )^2|\, dxd\xi
+\nm f{L^2}^2\log \nm f{L^2}^2
\to 0
\\[1ex]
\text{as}\quad
\nm f{M^\Phi}\to 0,\ f\in M^\Phi (\rr d),
\end{gathered}
\end{equation*}
and the asserted continuity for $E_\phi$ near origin follows.

\par


Next we prove that $E_\phi$ is continuous at a
general $f\in M^\Phi (\rr d)$. Due to the first part it suffices to
prove that $E_\phi$ is continuous outside origin. Therefore
assume that $f\in M^\Phi (\rr d)\setminus 0$. By using
the homogeneity
$E_\phi (\lambda f) = |\lambda |^2E_\phi (f)$ when $f\in M^\Phi (\rr d)$
in combination with \eqref{Eq:NarrowSpaces1},
it follows that it suffices to prove the result under the additional condition
$$
\nm f{M^\Phi}+\nm f{M^2}+\nm f{M^\infty} < 1.
$$

\par

For conveniency we set
$G=V_\phi g$, $H=F+G$ and
$$
J(f,g)
\equiv
\left |
\iint _{\rr {2d}}
\left (
|H(x,\xi )|^2\log |H(x,\xi )|
-|F(x,\xi )|^2\log |F(x,\xi )|
\right )
\, dxd\xi
\right |
$$
when $g\in M^\Phi (\rr d)$.
We have
\begin{equation}\label{Eq:L2LogTerms}
|E_\phi (f+g) -E_\phi (f)|
\le
2
\left (J(f,g)
+
\left |
\nm {H}{L^2}^2\log \nm {H}{L^2}
- \nm {F}{L^2}^2\log \nm {F}{L^2}
\right | 
\right ).
\end{equation}

\par

If $\nm g{M^\Phi}\to 0$, then $\nm g{M^2}\to 0$,
which implies that $\nm H{L^2}\to \nm F{L^2}$
as $\nm g{M^2}\to 0$. Hence, by the continuity
of $t^2\log t$ on $[0,\infty )$, it follows that
last modulus in \eqref{Eq:L2LogTerms} tends to
zero as $\nm g{M^\Phi}\to 0$. This implies that
the asserted continuity follows if we prove
\begin{equation}\label{Eq:J(f,g)ToZero}
J(f,g)\to 0
\quad \text{as}\quad
\nm g{M^\Phi}\to 0.
\end{equation}

\par

Let $R>1$ be fixed and let
$$
\Omega = \sets {(x,\xi )\in \rr {2d}}{|F(x,\xi )|>R|G(x,\xi )|}.
$$
Then
\begin{equation}\label{Eq:DecompJ(f,g)}
0\le J(f,g) \le \sum _{k=1}^3 J_k(f,g),
\end{equation}
where
\begin{align*}
J_1(f,g)
&=
\left |
\iint _{\Omega}(|H(x,\xi )|^2-|F(x,\xi )|^2)\log |F(x,\xi )|\, dxd\xi 
\right |,
\\[1ex]
J_2(f,g)
&=
\left |
\iint _{\Omega}|H(x,\xi )|^2\log \left |\frac {H(x,\xi )}{F(x,\xi )}
\right | \, dxd\xi 
\right |,
\intertext{and}
J_3(f,g)
&=
\left |
\iint _{\complement \Omega}
\left (
|H(x,\xi )|^2\log |H(x,\xi )|
-|F(x,\xi )|^2\log |F(x,\xi )|
\right )
\, dxd\xi
\right | .
\end{align*}
We shall estimate $J_k(f,g)$ in suitable ways, $k=1,2,3$.

\par

For the integrand in $J_1(f,g)$, taken into account that
$$
R|G(x,\xi )|<|F(x,\xi )|<1,
$$
we have
\begin{align*}
0&\le \left |
(|H(x,\xi )|^2-|F(x,\xi )|^2)\log |F(x,\xi )|
\right |
\\[1ex]
&=
- \big |
|F(x,\xi )+G(x,\xi )|^2-|F(x,\xi )|^2
\big | \log |F(x,\xi )|
\\[1ex]
&\le
-(2|F(x,\xi )|\, |G(x,\xi )| +|G(x,\xi )|^2)\log |F(x,\xi )|
\\[1ex]
&\le
-\left (\frac 2R+\frac 1{R^2}\right )|F(x,\xi )|^2\log |F(x,\xi )|,
\end{align*}
which gives
\begin{equation}\label{Eq:J1(f,g)Est}
J_1(f,g) \le -\left (\frac 2R+\frac 1{R^2}\right )
\iint _{\rr {2d}}|F(x,\xi )|^2\log |F(x,\xi )|\, dxd\xi .
\end{equation}

\par

For the logarithm in $J_2(f,g)$ we have
\begin{align*}
\left |
\log
\left |
\frac {H(x,\xi )}{F(x,\xi )}
\right |
\right |
&=
\left |
\log
\left |
1+\frac {G(x,\xi )}{F(x,\xi )}
\right |
\right |
\le
-\log
\left (
1-\frac {|G(x,\xi )|}{|F(x,\xi )|}
\right )
\\[1ex]
&\le
-\log
\left (
1-\frac 1R
\right )
=
\sum _{j=1}^\infty \frac {R^{-j}}j
\le
\sum _{j=1}^\infty R^{-j}
=
\frac 1{R-1}.
\end{align*}
In the second inequality we have used the fact that
$R>1$ and that $|F(x,\xi )|>R|G(x,\xi )|$ when $(x,\xi )\in \Omega$.

\par

This gives
\begin{align*}
J_2(f,g)
&\le
\frac 1{R-1}
\iint _{\Omega} |H(x,\xi )|^2\, dxd\xi 
\\[1ex]
&\le
\frac 2{R-1}
\iint _{\Omega} (|F(x,\xi )|^2+|G(x,\xi )|^2)\, dxd\xi 
\\[1ex]
&<
\frac 2{R-1}\iint _{\Omega} (|F(x,\xi )|^2+\frac 1{R^2}|F(x,\xi )|^2)\, dxd\xi ,
\end{align*}
which in turn gives
\begin{equation}\label{Eq:J2(f,g)Est}
J_2(f,g)
<
\frac 2{R-1}\left ( 1+\frac 1{R^2}\right )\nm F{L^2}^2
\end{equation}

\par

Next we estimate $J_3(f,g)$.
By \eqref{Eq:NarrowSpaces1} there is a
$\delta _0>0$ such that
\begin{equation}\label{Eq:GSpecEst}
|G(x,\xi )|\le \frac {e^{-\frac 12}}{(R+1)},
\quad \text{when}\quad
\nm g{M^\Phi}<\delta _0.
\end{equation}

\par

Since $|t^2\log t|=-t^2\log t$ is increasing on $[0,e^{-\frac 12}]$,
\begin{align*}
|H(x,\xi )|&\le |F(x,\xi )|+|G(x,\xi )|\le (R+1)|G(x,\xi )|\le e^{-\frac 12} 
\intertext{and}
|F(x,\xi )|&\le R|G(x,\xi )|\le e^{-\frac 12}
\end{align*}
when $(x,\xi )\in \complement \Omega$
by \eqref{Eq:GSpecEst}, we obtain
\begin{align*}
J_3(f,g)
&\le
\left |
\iint _{\complement \Omega}
|H(x,\xi )|^2\log |H(x,\xi )|\, dxd\xi 
\right |
+
\left |
\iint _{\complement \Omega}
|F(x,\xi )|^2\log |F(x,\xi )|
\, dxd\xi
\right | 
\\[1ex]
&\le
\left |
\iint _{\complement \Omega}
|(R+1)G(x,\xi )|^2\log |(R+1)G(x,\xi )|\, dxd\xi 
\right |
\\
&+
\left |
\iint _{\complement \Omega}
|RG(x,\xi )|^2\log |RG(x,\xi )|
\, dxd\xi
\right | 
\\[1ex]
&\le
\left (
(R+1)^2+R^2
\right )
\left |
\iint _{\complement \Omega}
|G(x,\xi )|^2\log |G(x,\xi )|\, dxd\xi 
\right |
\\
&+
\left (
(R+1)^2\log (R+1)+R^2\log R
\right )
\iint _{\complement \Omega}
|G(x,\xi )|^2\, dxd\xi 
\end{align*}
when $\nm g{M^\Phi}<\delta _0$.
A combination of these estimates and the fact that
$\log |G(x,\xi )|<0$ in view of \eqref{Eq:GSpecEst}
gives
\begin{equation}\label{Eq:J3(f,g)Est}
\begin{aligned}
J_3(f,g)
&\le 
-\left (
(R+1)^2+R^2
\right )
\iint _{\rr {2d}}
|G(x,\xi )|^2\log |G(x,\xi )|\, dxd\xi 
\\[1ex]
&+
\left (
(R+1)^2\log (R+1)+R^2\log R
\right ) \nm G{L^2}^2,
\quad \nm g{M^\Phi}<\delta _0.
\end{aligned}
\end{equation}

\par

Now let $\ep >0$ be arbitrary. By Lemma
\ref{Lem:EntropyCont}, 
\eqref{Eq:NarrowSpaces1},
\eqref{Eq:J1(f,g)Est}
and \eqref{Eq:J2(f,g)Est} it follows that
$J_1(f,g)<\frac \ep 3$ and $J_2(f,g)<\frac \ep 3$,
provided $R$ is chosen large enough.
A combination of Lemma \ref{Lem:PartEntropyCont},
\eqref{Eq:NarrowSpaces1} and
\eqref{Eq:J3(f,g)Est} shows that
there is a positive number $\delta <\delta _0$
such that $J_3(f,g)<\frac \ep 3$ when
$\nm g{M^\Phi}<\delta$.

\par

By combining these estimates with
\eqref{Eq:DecompJ(f,g)} now gives
$$
0\le J(f,g)<\ep
\quad \text{when}\quad
g\in M^\Phi (\rr d),\ \nm g{M^\Phi}<\delta .
$$
This shows that \eqref{Eq:J(f,g)ToZero} holds true, and the
continuity for $E_\phi$ on $M^\Phi (\rr d)$ follows.

\par

The continuity for $E_\phi$ on $M^p(\rr d)$, $1\le p<2$
now follows from the fact that 
$M^p (\rr d)$ is continuously embedded in
$M^\Phi (\rr d)$.

\par

It remains to prove the discontinuity for $E_\phi$ on $M^2(\rr d)$,
and then it follows from Lemma \ref{Lem:DiscontEntrOnL2}
that we may assume that $\phi (x)=\pi ^{-\frac d4}e^{-\frac 12|x|^2}$.
We shall investigate $E_\phi (f)$ with
$$
f(x)=f_\lambda (x)=\pi ^{-\frac d4}\lambda ^{\frac d4}e^{-\frac \lambda 2|x|^2},
\qquad \lambda >1.
$$
Then $\nm {\phi}{L^2}=\nm {f_\lambda}{L^2}=1$, and by straight-forward
computations it follows that
$$
V_\phi f_\lambda (x,\xi )
=
\left (\frac {\lambda ^{\frac 12}} {\pi (\lambda +1)}\right )^{\frac d2}
e^{-\frac i{\lambda +1}\scal x\xi}
e^{-\frac 1{2(\lambda +1)}(\lambda |x|^2+|\xi |^2)},
$$
and since $f_\lambda$ is $L^2$-normalized we get
\begin{align*}
E_\phi (f_\lambda ) &= -\iint _{\rr {2d}}|V_\phi f_\lambda (x,\xi )|^2
\log |V_\phi f_\lambda (x,\xi )|^2\, dxd\xi
\\[1ex]
&=
\left (\frac {\lambda ^{\frac 12}} {\pi (\lambda +1)}\right )^{d}
\iint _{\rr {2d}} h_\lambda \Big ( \frac 1{\lambda +1}
(\lambda |x|^2+ |\xi |^2 ) \Big )\, dxd\xi ,
\end{align*}
where
$$
h_\lambda (t) = e^{-t}\left ( \frac t2+d
\log \left ( \pi (\lambda ^{\frac 12}+\lambda ^{-\frac 12})\right )
\right )
$$
By taking $(\frac \lambda {\lambda +1})^{\frac 12}x$ and
$(\frac 1{\lambda +1})^{\frac 12}\xi $ as new variables of
integrations we obtain
\begin{align*}
E_\phi (f_\lambda )
&=
\pi ^{-d}\iint _{\rr {2d}} h_\lambda ( |x|^2+ |\xi |^2)\, dxd\xi
\\[1ex]
&=
\pi ^{-d}
\iint _{\rr {2d}} e^{-(|x|^2+|\xi |^2)}
\left (
\frac 1{2}(|x|^2+|\xi |^2)
+
d\log \left ( \pi (\lambda ^{\frac 12}+\lambda ^{-\frac 12})\right )\right )
\, dxd\xi
\\[1ex]
&=
d\left (
\frac 14
+
\log \left ( \pi (\lambda ^{\frac 12}+\lambda ^{-\frac 12})\right )
\right ) .
\end{align*}

\par

This implies
$$
\lim _{\lambda \to 0+}E_\phi (f_\lambda )
=
\lim _{\lambda \to \infty}E_\phi (f_\lambda ) = \infty
\quad \text{but}\quad
\nm {f_\lambda}{L^2}=1,
$$
which shows that $E_\phi$ is discontinuous on $L^2(\rr d)=M^2(\rr d)$,
and the result follows.
\end{proof}

\par

By Theorem \ref{Thm:EntropyCont} and its proof
it follows that Lemma \ref{Lem:PartEntropyCont}
can be improved into the following.

\par

\renewcommand{\rubrik}{Lemma \ref{Lem:PartEntropyCont}$'$\!}

\par

\begin{tom}
Let $\Phi$ and $\phi$ be the same as in
Lemma \ref{Lem:EntropyCont}. Then
$$
M^\Phi (\rr d)\ni f\mapsto \iint _{\rr {2d}}
|V_\phi f(x,\xi )|^2 \big | \log |V_\phi f(x,\xi )| \big |\, dxd\xi
$$
is locally uniformly continuous.
\end{tom}

\par

\begin{rem}
In view of Theorem \ref{Thm:EntropyCont} and its
proof it follows that (2) in that theorem can be extended
into the following:
\begin{enumerate}
\item[(2)$'$] \emph{$E_\phi$ in \eqref{Eq:EntropyFunc}
is locally uniformly continuous on $M^p(\rr d)$ and on $M^\Phi (\rr d)$,
$1\le p<2$,
and discontinuous on $M^2(\rr d)$.}
\end{enumerate}
\end{rem}

\par

\begin{rem}\label{Rem:DiscontPseudo}
Let $A\in \GL (d,\mathbf R)$ and $\Phi$ be a Young
function which fullfils \eqref{Eq:SpecYoungFunc}.
We claim that there is
a symbol $a$ in $M^2(\rr {2d})$ (which is close
to $M^{p,p'}(\rr {2d})$ when $p>2$ is close to $2$)
such that  the map \eqref{Eq:ExContPseudo} is
discontinuous. (Cf. Example \ref{Example:ContPseudo}.)

\par

In fact, by Remark \ref{Rem:NonEqualOrlModSp},
there are $f_1\in M^2(\rr d)\setminus M^\Phi (\rr d)$
and $f_2\in \mascS (\rr d)\setminus 0$. Then
$a=W_{f_1,f_2}^A \in M^2(\rr {2d})$.
By \eqref{Eq:PseudoWignerLink2} it follows that
$$
\op _A(a)f(x) = (2\pi )^{-\frac d2}(f,f_2)_{L^2}f_1(x)
\in M^2(\rr d)\setminus M^\Phi (\rr d)
$$
for every $f\in \mascS (\rr d)\subseteq M^\Phi (\rr d)$ which
is not orthogonal to $f_2$, and the asserted discontinuity
follows. 
\end{rem}

\par

\appendix


\section{STFT Projections
on Orlicz modulation spaces}\label{app:A}

\par

In this appendix we first recall some facts on projections on
Orlicz modulation spaces which appear after compositions
between short-time Fourier transforms and their adjoints.

\par

Thereafter we apply our results to give a proof 
of Proposition \ref{Prop:BasicPropOrlModSp2}.

\par

\subsection{STFT Projections and
twisted convolutions}

\par

Let $s\ge \frac 12$.
If $\phi \in \maclS _s(\rr d)\setminus 0$,
then it follows from Fourier's inversion
formula that
\begin{equation}\label{Eq:IdentSTFTAdj}
\operatorname{Id}
=
\operatorname{Id}_{\maclS _s'}
=
\left (\nm \phi{L^2}^{-2}\right ) \cdot V_\phi ^*\circ V_\phi ,
\end{equation}
is the identity operator on
$\maclS _s'(\rr d)$. The same and
following results hold
true with $\Sigma _s$ and $\mascS$
in place of $\maclS _s$ at each occurrence.
The identity \eqref{Eq:IdentSTFTAdj}
is equivalent to Moyal's identity
\eqref{Eq:Moyal}. If we swap the
order of this composition we get certain types
of projections. More precisely, let
\begin{equation}\label{Eq:ProjphiDef}
P_\phi \equiv \nm \phi {L^2}^{-2}
\cdot V_\phi \circ V_\phi ^*.
\end{equation}
We observe that $P_\phi$ is continuous on
$\maclS_s (\rr {2d})$,
$L^2(\rr {2d})$ and on $\maclS _s'(\rr {2d})$
due to the mapping properties
for $V_\phi$ and $V_\phi ^*$.

\par

It is clear that $P_\phi ^*=P_\phi$, i.{\,}e. 
$P_\phi$ is self-adjoint. Furthermore,
$$
P_\phi ^2 = \nm \phi {L^2}^{-2}\cdot
V_\phi \circ
\Big (
\underset{\text{The identity operator}}
{\underbrace{\nm \phi {L^2}^{-2}\cdot V_\phi ^*\circ V_\phi}}
\Big )
\circ V_\phi ^*
=
\nm \phi {L^2}^{-2}\cdot V_\phi \circ V_\phi ^* =P_\phi ,
$$
giving that $P_\phi$ is an orthonormal projection,
that is,
\begin{equation}\label{Eq:ProjphiRule}
P_\phi ^* = P_\phi
\quad \text{and}\quad P_\phi ^2=P_\phi .
\end{equation}

\par

The ranks of $P_\phi$ are given by
\begin{equation}\label{Eq:STFTProjMaps}
\begin{aligned}
P_\phi (\maclS _s (\rr {2d}))
=
V_\phi (\maclS _s (\rr d))
\quad \text{and}\quad
P_\phi (\maclS _s '(\rr {2d}))
=
V_\phi (\maclS _s'(\rr d)).
\end{aligned}
\end{equation}
In fact, if $F\in \maclS _s '(\rr {2d})$, then
\begin{equation}\label{Eq:ProjSTFTIdent}
P_\phi F = V_\phi f,
\end{equation}
where $f=\nm {\phi}{L^2}^{-2}V_\phi ^*F\in
\maclS _s '(\rr d)$. This shows that
$P_\phi (\maclS _s '(\rr {2d}))
\subseteq V_\phi (\maclS _s'(\rr d))$.
On the other hand, if
$f\in \maclS _s'(\rr d)$ and $F=V_\phi f$, then
$$
P_\phi F =   \Big (V_\phi \circ \Big ( \nm \phi {L^2}^{-2} \cdot
V_\phi ^* \circ V_\phi \Big ) \Big )f = V_\phi f,
$$
which shows that any element in
$V_\phi (\maclS _s '(\rr d))$
equals to an element in
$P_\phi (\maclS _s'(\rr {2d}))$,
i.{\,}e. $P_\phi (\mascS '(\rr {2d})) = V_\phi (\mascS '(\rr d))$.
The same holds true
with $\maclS _s$ in place of $\maclS _s'$
at each occurrence, and
\eqref{Eq:STFTProjMaps} follows.

\par

\begin{rem}\label{Rem:STFTProj}
Let $F\in \maclS _s'(\rr {2d})$.
Then \eqref{Eq:STFTProjMaps} shows
that $F=V_{\phi}f$ for some
$f\in \maclS _s'(\rr {d})$, if and only if
\begin{equation}\label{Eq:TwistedProj}
F= P_\phi F.
\end{equation}
Furthermore, if \eqref{Eq:TwistedProj} holds, then $F=V_{\phi}f$ with
\begin{equation}\label{Eq:TwistedProj2}
f=(\nm \phi{L^2}^{-2})\cdot V_\phi ^*F .
\end{equation}
%
%
%
%
\end{rem}

\medspace

Let $F\in \maclS _s (\rr {2d})$
and $\phi \in \maclS _s (\rr d)\setminus 0$.
Then by expanding the integrals
for $V_\phi$ and $V_\phi ^*$ in
\eqref{Eq:ProjphiDef} one obtains
\begin{equation}\label{Eq:ProjOpTwistedConv}
P_\phi F = \nm \phi {L^2}^{-2}
\cdot
V_\phi \phi *_V F ,
\qquad F\in \mascS '(\rr {2d}),
\end{equation}
where the
\emph{twisted convolution} $*_V$ is defined by
\begin{align}
(F*_VG)(x,\xi )
&=
(2\pi )^{-\frac d2}
\iint _{\rr {2d}}
F(x-y,\xi -\eta )G(y,\eta )e^{-i\scal{y}{\xi -\eta}}\, dyd\eta ,
\label{Eq:TwistConvDef}
\end{align}
when $F,G\in \maclS _s (\rr {2d})$.
We observe that the definition of $*_V$
extends in different ways.
For example, Young's inequality for ordinary 
convolution also holds for $*_V$.
Moreover, the map $(F,G)\mapsto F*_VG$
extends uniquely to continuous mappings from
$\maclS _s(\rr {2d}) \times
\maclS _s '(\rr {2d})$ or
$\maclS _s '(\rr {2d})
\times \maclS _s (\rr {2d})$
to $\maclS _s '(\rr {2d})$. By straight-forward 
computations it follows that
\begin{equation}\label{Eq:TwistedConvAsoc}
(F*_VG)*_VH = F*_V(G*_VH),
\end{equation}
when $F,H\in \maclS _s (\rr {2d})$ and
$G\in \maclS _s '(\rr {2d})$,
or $F,H\in \maclS _s '(\rr {2d})$
and $G\in \maclS _s(\rr {2d})$

\par

Let $f\in \maclS _s'(\rr d)$ and
$\phi _j\in \maclS (\rr d)$, $j=1,2,3$.
By straight-forward applications of
Parseval's formula it follows that
\begin{equation}\label{Eq:STFTWindTrans}
\big (
(V_{\phi _2}\phi _3) *_V(V_{\phi _1}f)
\big ) (x,\xi )
=
(\phi _3,\phi _1)_{L^2} \cdot (V_{\phi _2}f)(x,\xi ),
\end{equation}
which is some sort of reproducing kernel
of short-time Fourier transforms in the background of
$*_V$. (See also Chapter 11 in \cite{Gro1}.)

\par

\subsection{Applications to Orlicz
modulation spaces}

\par

We have now the following which essentially follows
from Proposition 4.3 and its proof in
\cite{FeiGro1}.

\par

\begin{lemma}\label{Lemma:ContProjOrlSp}
Let $\Phi$ and $\Psi$ be Young functions,
$\phi \in \Sigma _1(\rr d)$ be such that
$\nm \phi{L^2}=1$ and let $\omega \in
\mascP _E(\rr {2d})$. Then
the following is true:
\begin{enumerate}
\item
$P_\phi$ from $\Sigma _1'(\rr {2d})$
to $V_\phi (\Sigma _1'(\rr {d}))$ restricts
to a continuous projection from
$L^{\Phi ,\Psi }_{(\omega )}(\rr {2d})$
to $V_\phi
(M^{\Phi ,\Psi }_{(\omega )}(\rr {d}))$;

\vrum

\item if $F\in
L^{\Phi ,\Psi }_{(\omega )}(\rr {2d})$
and $f=V_\phi ^*F$, then $V_\phi f=P_\phi F$
and
\begin{equation}\label{Eq:ProjOrlModCont}
\nm f{M^{\Phi ,\Psi }_{(\omega )}}
\asymp
\nm {P_\phi F}{L^{\Phi ,\Psi }_{(\omega )}}
\lesssim
\nm F{L^{\Phi ,\Psi }_{(\omega )}},
\qquad
f=V_\phi ^*F .
\end{equation}
\end{enumerate}
\end{lemma}

\par

\begin{proof}
By \eqref{Eq:ProjSTFTIdent}
and Remark \ref{Rem:STFTProj}, the
result follows if we prove
\eqref{Eq:ProjOrlModCont}.

\par

Let $v\in \mascP _E(\rr {2d})$ be submultiplicative
such that $\omega$ is $v$-moderate. By 
\eqref{Eq:ProjOpTwistedConv} we have
$$
|F*_VG| \le |F|*|G|.
$$
Hence \eqref{Eq:ProjSTFTIdent},
and \eqref{Eq:TwistConvDef} give
\begin{align*}
\nm f{M^{\Phi ,\Psi }_{(\omega )}}
&\asymp
\nm {V_\phi f}{L^{\Phi ,\Psi }_{(\omega )}}
=
\nm {P_\phi F}{L^{\Phi ,\Psi }_{(\omega )}}
\\[1ex]
&\le
\NM {|F|*|V_\phi \phi|}{L^{\Phi ,\Psi }_{(\omega )}}
\lesssim
\nm F{L^{\Phi ,\Psi }_{(\omega )}}
\nm {V_\phi \phi}{L^1_{(v)}}.
\end{align*}
The asserted continuity now follows from
the fact that for some $r>0$ we have
$$
v(x,\xi )\lesssim e^{r(|x|+|\xi |)}
\quad \text{and}\quad
|V_\phi \phi (x,\xi )| \lesssim 
e^{-2r(|x|+|\xi |)},
$$
in view of Proposition \ref{stftGelfand2}
and \eqref{Eq:BoundWeights}.
\end{proof}

\par

\begin{proof}[Proof of Proposition
\ref{Prop:BasicPropOrlModSp2}]
We have
$$
|(F,G)_{L^2(\rr {2d})}|
\lesssim
\nm F{L^{\Phi ,\Psi}_{(\omega )}}
\nm G{L^{\Phi ^*,\Psi ^*}_{(\omega )}}
$$
when $F,G\in \Sigma _1(\rr {2d})$, by H{\"o}lder's
inequality for Orlicz spaces (cf. e.{\,}g. \cite{HaH,RaoRen1}).
By Hahn-Banach's theorem it follows that the map
$(F,G)\to (F,G)_{L^2(\rr {2d})}$ from $\Sigma _1(\rr {2d})
\times \Sigma _1(\rr {2d})$ to $\mathbf C$ extends to
a continuous map from $L^{\Phi ,\Psi}_{(\omega )}(\rr {2d})
\times L^{\Phi ^*,\Psi ^*}_{(\omega )}(\rr {2d})$ to
$\mathbf C$.

\par

If $\phi \in \Sigma _1(\rr d)\setminus 0$
satisfies $\nm \phi {L^2}=1$,
$f\in M^{\Phi ,\Psi}_{(\omega )}(\rr {d})$ and
$g\in M^{\Phi ^*,\Psi ^*}_{(\omega )}(\rr {d})$,
we now use Moyal's identity to define
$(f,g)_{L^2(\rr d)} = (V_\phi f,V_\phi g)_{L^2(\rr {2d})}$,
which satisfies the requested properties, because
\begin{equation}\label{Eq:AppDualityEst}
\begin{aligned}
|(f,g)_{L^2(\rr d)}| &= |(V_\phi f,V_\phi g)_{L^2(\rr {2d})}|
\\[1ex]
&\lesssim
\nm {V_\phi f}{L^{\Phi ,\Psi}_{(\omega )}}
\nm {V_\phi g}{L^{\Phi ^*,\Psi ^*}_{(\omega )}}
\asymp
\nm {f}{M^{\Phi ,\Psi}_{(\omega )}}
\nm {g}{M^{\Phi ^*,\Psi ^*}_{(\omega )}},
\end{aligned}
\end{equation}
and the continuity extension in (1) follows. Suppose from now
on that $\Phi$ and $\Psi$ in addition
satisfy the $\Delta _2$-condition. Then
$\Sigma _1(\rr d)$ is dense in $M^{\Phi ,\Psi}_{(\omega )}(\rr d)$
which implies that the latter continuity extension is unique.

\par

%

Next suppose that $T$ is a continuous linear form on
$M^{\Phi ,\Psi}_{(\omega )}(\rr d)$. Then
$$
T_1(V_\phi f) \equiv T(f)
$$
satisfies
$$
|T_1(V_\phi f)| \lesssim |T(f)|
\lesssim
\nm f{M^{\Phi ,\Psi}_{(\omega )}}
\asymp
\nm {V_\phi f}{L^{\Phi ,\Psi}_{(\omega )}}.
$$
Hence $T_1$ is a continuous linear form on
$V_\phi (M^{\Phi ,\Psi}_{(\omega)}(\rr d))$.
Since the injection from
$V_\phi (M^{\Phi ,\Psi}_{(\omega)}(\rr d))$
to $L^{\Phi ,\Psi}_{(\omega)}(\rr {2d}))$
is norm preserving,
it follows by Hahn-Banach's theorem that
$T_1$ extends to a linear form on
$L^{\Phi ,\Psi}_{(\omega)}(\rr {2d}))$ with the same
norm.
By \cite{RaoRen1}
it follows that the dual of the latter space is equal to
$L^{\Phi ^*,\Psi ^*}_{(1/\omega)}(\rr {2d})$
through the $(\cdo ,\cdo )_{L^2(\rr {2d})}$ form. Hence
$$
T_1(F) = (F,G)_{L^2(\rr {2d})}
=
\iint _{\rr {2d}} F(x,\xi )\overline {G(x,\xi )}\, dxd\xi ,
\qquad F\in L^{\Phi ,\Psi}_{(\omega)}(\rr {2d}),
$$
for some fixed $G\in L^{\Phi ^*,\Psi ^*}_{(1/\omega)}(\rr {2d}))$
which satisfies
\begin{equation}\label{Eq:DualElemEsts}
\nm G{L^{\Phi ^*,\Psi ^*}_{(1/\omega)}}\asymp \nmm {T_1}=\nmm T.
\end{equation}

\par

By Lemma \ref{Lemma:ContProjOrlSp} we also have
$P_\phi G=V_\phi g$
for some $g\in M^{\Phi ,\Psi}_{(\omega)}(\rr {d})$.
A combination of these identities and Moyal's identity
gives that for any $f\in M^{\Phi ,\Psi}_{(\omega)}(\rr {d})$
we have
\begin{equation}\label{Eq:ReformTOp}
\begin{aligned}
T(f) = (V_\phi f,G)_{L^2(\rr {2d})}
&=
(P_\Phi (V_\phi f),G)_{L^2(\rr {2d})}
\\[1ex]
&=
(V_\phi f,P_\phi G)_{L^2(\rr {2d})}
=
 (V_\phi f,V_\phi g)_{L^2(\rr {2d})}
 =
 (f,g)_{L^2((\rr d))},
\end{aligned}
\end{equation}
which gives (2).

\par

Finally, by
\eqref{Eq:AppDualityEst} it follows that
$\nmm f \lesssim \nm f {M^{\Phi ,\Psi}_{(\omega )}}$ when
$f\in M^{\Phi ,\Psi}_{(\omega )}(\rr d)$.

\par

On the other hand, let $f_0\in M^{\Phi ,\Psi}_{(\omega )}(\rr d)$
be fixed and let $T$ be the linear form on
$\sets {\lambda f_0}{\lambda \in \mathbf C}
\subseteq
M^{\Phi ,\Psi}_{(\omega )}(\rr d)$ given by
$$
T(\lambda f_0)=\lambda \nm {f_0}{M^{\Phi ,\Psi}_{(\omega )}}.
$$
Then $\nmm T=1$. By Hahn-Banach's theorem, there is a
$G\in L^{\Phi ^*,\Psi ^*}_{(1/\omega)}(\rr {2d}))$ such that
$T$ extends to a form on $M^{\Phi ,\Psi}_{(\omega )}(\rr d)$
and such that
\eqref{Eq:DualElemEsts} and \eqref{Eq:ReformTOp} hold.
Since $\nm g{M^{\Phi ^*,\Psi ^*}_{(1/\omega )}}\lesssim
\nm G{L^{\Phi ^*,\Psi ^*}_{(1/\omega )}}$ in view of Lemma
\ref{Lemma:ContProjOrlSp} we get by choosing $f=f_0$ that
$$
\nm {f_0}{M^{\Phi ,\Psi}_{(\omega )}}=T(f_0)=(f_0,g)_{L^2}
\lesssim \sup |(f_0,g)_{L^2}|=\nmm {f_0},
$$
where the hidden constants are independent of
$f_0\in M^{\Phi ,\Psi}_{(\omega )}(\rr d)$. Here the
supremum is taken over all
$g\in M^{\Phi ^*,\Psi ^*}_{(1/\omega )}(\rr d)$
such that $\nm g{M^{\Phi ^*,\Psi ^*}_{(1/\omega )}}\le 1$.
Consequently we have $\nm {f}{M^{\Phi ,\Psi}_{(\omega )}}
\asymp \nmm {f_0}$, giving that (1), and thereby
the result follow.
%
%
%
%
%
%
%
%
%
%
%
%
%
%
%
%
%
%
%
%
%
\end{proof}


\par


\begin{thebibliography}{99}
\bibitem{CapTof} 
M. Cappiello, J. Toft
\emph{Pseudo-differential operators
in a Gelfand--Shilov setting},
Math. Nachr. \textbf{290} (2017), 738--755. 

\bibitem{CorGia} E. Cordero, G. Giacchi
\emph{Symplectic analysis of time-frequency spaces},
J. Math. Pure Appl. (2023), 154--177. 

\bibitem{CorNic} E. Cordero, F. Nicola
\emph{Pseudodifferential operators on
$L^p$, Wiener amalgam and modulation spaces},
Int. Math. Res. Not. IMRN \textbf{8} (2010), 1860--1893. 


\bibitem{CorNic2} E. Cordero, F. Nicola
\emph{Sharp integral bounds for Wigner distributions},
Int. Math. Res. Not. \textbf{6} (2018), 1779--1807.


\bibitem{CorOko1} E. Cordero, K. Okoudjou
\emph{Multilinear localization operators},
J. Math. Anal. Appl. \textbf{325} (2007),
1103--1116.

\bibitem{CPRT10} E. Cordero, S. Pilipovi\'c, L. Rodino, N. Teofanov
\emph{Quasianalytic Gelfand-Shilov spaces with applications
to localization operators}, Rocky Mt. J. Math. \textbf{40} (2010),
1123--1147.

\bibitem{CorRod1}
E. Cordero, L. Rodino
\emph{Time-frequency analysis of operators},
De Gruyter Studies in Mathematics,
\textbf{75}, De Gruyter, Berlin, 2020.

\bibitem{F1}  H. G. Feichtinger
\emph{Modulation spaces on locally
compact abelian groups. Technical report}, {University of
Vienna}, Vienna, 1983; also in: M. Krishna, R. Radha,
S. Thangavelu (Eds) Wavelets and their applications, Allied
Publishers Private Limited,
NewDelhi Mumbai Kolkata Chennai Nagpur
Ahmedabad Bangalore Hyderabad Lucknow, 2003, pp. 99--140.

\bibitem{Fei5} H. G. Feichtinger \emph{Modulation spaces: Looking
back and ahead}, Sampl. Theory Signal Image Process. \textbf{5}
(2006), 109--140.


\bibitem{FeiGro1}  {H. G. Feichtinger,  K.  Gr{\"o}chenig}
\emph{Banach spaces related to integrable group representations and
their atomic decompositions, I}, J. Funct. Anal. \textbf{86}
(1989), 307--340.

\bibitem{FeiGro2} {H. G. Feichtinger, K. Gr{\"o}chenig}
\emph{Banach spaces related to
integrable group representations and their atomic decompositions, II},
Monatsh. Math., \textbf{108} (1989), 129--148.

\bibitem{Fol1}  G. B. Folland
\emph{Harmonic analysis in phase space},
{Princeton U. P., Princeton}, 1989.

\bibitem{GaSa} Y. V. Galperin, S. Samarah \emph{Time-frequency analysis
on modulation spaces $M^{p,q}_m$, $0<p,q\le \infty$}, Appl. Comput.
Harmon. Anal. \textbf{16} (2004), 1--18.

\bibitem{GS} I. M. Gelfand and  G. E. Shilov
\emph{Generalized functions, II-III},
Academic Press, NewYork London, 1968.

\bibitem{Gro1} {K. Gr{\"o}chenig}, \emph{Foundations of
time-frequency analysis},
Birkh{\"a}user, Boston, 2001.

\bibitem{Gro2.5} K. Gr\"{o}chenig
\emph{Composition and spectral invariance
of pseudodifferential operators on modulation spaces},
J. Anal. Math. \textbf{98} (2006), 65--82.


\bibitem{GH1}  {K. Gr{\"o}chenig and C. Heil} \emph
{Modulation spaces and pseudo-differential operators},
Integral Equations Operator Theory (4) \textbf{34} (1999), 439--457.

\bibitem{GH2} {K. Gr{\"o}chenig and C. Heil} \emph {Modulation spaces
as symbol classes for pseudodifferential operators {\rm {in: M. Krishna, R. Radha,
S. Thangavelu (Eds) Wavelets and their applications}}}, Allied
Publishers Private Limited, NewDelhi Mumbai Kolkata Chennai Nagpur
Ahmedabad Bangalore Hyderabad Lucknow, 2003, pp. 151--170.

\bibitem{GroZim} K. Gr{\"o}chenig, G. Zimmermann
\emph{Spaces of test functions via the STFT},
J. Funct. Spaces Appl. \textbf{2} (2004), 25--53.

\bibitem{HaH} P. Harjulehto, P. H{\"a}st{\"o}
\emph{Orlicz spaces and generalized Orlicz spaces},
Springer, Cham, 2019.

\bibitem{Hor1}  L. H{\"o}rmander
\emph{The Analysis of Linear
Partial Differential Operators}, vol {I--III},
Springer-Verlag, Berlin Heidelberg NewYork Tokyo,
1983, 1985.

\bibitem{Kot1} G. K{\"o}the
\emph{Topological vector spaces I},
Die Grundlehren der mathematischen Wissenschaften \textbf{159},
Springer-Verlag, NewYork-Berlin, 1969.

\bibitem{Lie1} E. H. Lieb
\emph{Integral bounds for radar ambiguity functions
and Wigner distributions},
J. Math. Phys. \textbf{31} (1990), 594--599.

\bibitem{LieSol}  E. H. Lieb, J. P. Solovej
\emph{Quantum coherent operators: a generalization
of coherent states}
Lett. Math. Phys. \textbf{22} (1991), 145--154. 

\bibitem{MajLab1}
W. A. Majewski, L. E. Labuschagne
\emph{On applications of Orlicz spaces to statistical physics},
Ann. Henri Poincar{\'e} \textbf{15} (2014), 1197--1221.

\bibitem{MajLab2}
W. A. Majewski, L. E. Labuschagne
\emph{On entropy for general quantum systems},
Adv. Theor. Math. Phys. \textbf{24} (2020), 491--526.

\bibitem{Pil1}
S. Pilipovi\'c
\emph{Generalization of Zemanian spaces
of generalized functions which
have orthonormal series expansions},
SIAM J. Math. Anal. \textbf{17} (1986), 477--484.

\bibitem{Pil3}
S. Pilipovi\'c \emph{Tempered ultradistributions},
Boll. U.M.I. \textbf{7} (1988), 235--251.

\bibitem{RaoRen1}
M. M. Rao, Z. D. Ren
\emph{Theory of Orlicz Spaces},
Marcel Dekker, New York, 1991.

\bibitem{Rau1} H. Rauhut
\emph{Wiener amalgam spaces with respect
to quasi-Banach spaces},
Colloq. Math. \textbf{109} (2007), 345--362.

\bibitem{Rau2} H. Rauhut
\emph{Coorbit space theory for quasi-Banach
spaces}, Studia Math. \textbf{180} (2007), 237--253.

\bibitem{Rud} W. Rudin
\emph{Real and complex analysis},
McGraw-Hill Book Co., New York, 1987.

\bibitem{SchWol} H. H. Schaefer, M. P. Wolff
\emph{Topological vector spaces, 2nd. Ed.},
Graduate Texts in Mathematics, \textbf{3}, Springer-Verlag,
New York, 1999.

\bibitem{SchF} C. Schnackers, H. F{\"u}hr
\emph{Orlicz Modulation Spaces},
Proceedings of the 10th International Conference on 
Sampling Theory and Applications.

\bibitem{Teo1}
N. Teofanov
\emph{Ultradistributions and time-frequency analysis
{\rm {in: P. Boggiatto, L. Rodino, J. Toft, M. W. Wong (eds)}}
Pseudo-differential operators and related topics},
Operator Theory: Advances and Applications \textbf{164},
Birkh{\"a}user, Basel, 2006, pp. 173--192.

\bibitem{TeoTof} N. Teofanov, J. Toft
\emph{Pseudo-differential calculus in a
Bargmann setting},
 Ann. Acad. Sci. Fenn. Math. \text{45} (2020), 
 227--257.
 



\bibitem{Toft5} J. Toft
\emph{Continuity properties for modulation spaces, with
applications to pseudo-differential operators, I},
J. Funct. Anal. \textbf{207} (2004), 399--429.

\bibitem{Toft6} J. Toft
\emph{Continuity properties for modulation spaces, with
applications to pseudo-differential operators, II},
Ann. Glob. Anal. and Geom. \textbf{26} (2004), 73--106.

\bibitem{Toft8B} J. Toft
\emph{Pseudo-differential operators with symbols in
modulation spaces}, in: B.-W. Schulze, M. W. Wong (Eds),
Pseudo-Differential Operators: Complex Analysis and Partial
Differential Equations, Operator Theory
Advances and Applications
\textbf{205}, Birkh{\"a}user Verlag, Basel, 2010,
pp. 223--234.

\bibitem{Toft28} J. Toft
\emph{The Bargmann transform on modulation and
Gelfand-Shilov spaces, with applications to Toeplitz and
pseudo-differential operators}, J. Pseudo-Differ. Oper. Appl. 
\textbf{3} (2012), 145--227.

\bibitem{Toft15} J. Toft
\emph{Gabor analysis for a broad class of quasi-Banach modulation
spaces {\rm {in: S. Pilipovi{\'c}, J. Toft (eds)}},
Pseudo-differential operators, generalized functions},
Operator Theory: Advances and Applications \textbf{245},
Birkh{\"a}user, 2015, 249--278.

\bibitem{Toft16} J. Toft
\emph{Continuity and compactness for pseudo-differential
operator with symbols in quasi-Banach spaces or H{\"o}rmander classes},
Analysis and Applications,  \textbf{15} (2016), 353--389.


\bibitem{Toft15B} J. Toft
\emph{Matrix parameterized pseudo-differential
calculi on modulation spaces
{\rm {in: M. Oberguggenberger, J. Toft,
J. Vindas, P. Wahlberg (eds)}},
Generalized functions and Fourier analysis},
Operator Theory: Advances and Applications \textbf{260}
Birkh{\"a}user, 2017, pp. 215--235.

\bibitem{Toft18} J. Toft
\emph{Images of function and distribution spaces under the
Bargmann transform},
J. Pseudo-Differ. Oper. Appl. \textbf{8} (2017),
83--139.

\bibitem{TofUst} J. Toft, R. {\"U}ster
\emph{Pseudo-differential operators on Orlicz
modulation spaces},
J. Pseudo-Differ. Oper. Appl.
\textbf{14} (2023), Paper no. 6.

\bibitem{ToUsNaOz}  J. Toft, R. {\"U}ster,
E. Nabizadeh and S. {\"O}ztop
\emph{Continuity and Bargmann mapping
properties of quasi-Banach Orlicz modulation spaces}
Forum. Math. \textbf{34} (2022), 1205--1232.

\bibitem{Tra} G. Tranquilli \emph{Global normal
forms and global properties in function spaces for 
second order Shubin type operators},
PhD Thesis, 2013.
\end{thebibliography}
\end{document}